\documentclass[UTF8]{article}
\usepackage[margin=1.5in]{geometry}  
\usepackage{amssymb}             
\usepackage{amsmath}
\usepackage{mathtools}
\usepackage{slashed}            
\usepackage{amsfonts}              
\usepackage{amsthm}
\usepackage{breqn}
\usepackage{verbatim}              
\usepackage{enumitem}
\usepackage[colorlinks,linkcolor=black,anchorcolor=black,citecolor=black]{hyperref}

\usepackage[numbers]{natbib}
\usepackage{mathtools}
\usepackage{cases}
\usepackage{fancyhdr}

\theoremstyle{definition}
\newtheorem{thm}{Theorem}[section]
\newtheorem{nota}[thm]{Notation}

\newtheorem{defn}{Definition}[section]
\newtheorem{lem}[thm]{Lemma}
\newtheorem{prop}[thm]{Proposition}

\newcommand{\R}{\mathbb{R}}

\newcommand{\T}{\mathbb{T}}
\newcommand{\TP}{\overline{\partial}{}}

\newcommand{\curl}{\text{curl }}
\newcommand{\dive}{\text{div }}

\newcommand{\q}{\quad}
\newcommand{\p}{\partial}

\newcommand{\DD}{\mathcal{D}}

\newcommand{\nab}{\nabla}

\newcommand{\di}{\text{div}\,}

\newcommand{\cp}{\overline{\partial}{}}

\newcommand{\PP}{\mathcal{P}}

\numberwithin{equation}{section}
\usepackage{xcolor}

\begin{document}
\title{A priori Estimates for the Incompressible Free-Boundary Magnetohydrodynamics Equations with Surface Tension}
\author{Chenyun Luo\thanks{Department of Mathematics, The Chinese University of Hong Kong, Shatin, NT, Hong Kong. E-mail:
\texttt{cluo@math.cuhk.edu.hk}}\,\, and Junyan Zhang \thanks{Johns Hopkins University, Baltimore, MD 21218, USA. E-mail:
\texttt{zhang.junyan@jhu.edu}} }
\maketitle

\begin{abstract}
We consider the three-dimensional incompressible free-boundary magnetohydrodynamics (MHD) equations in a bounded domain with surface tension on the boundary. We establish a priori estimate for solutions in the Lagrangian coordinates with $H^{3.5}$ regularity. To the best of our knowledge, this is the first result focusing on the incompressible ideal free-boundary MHD equations with surface tension. It is worth pointing out that the $1/2$-extra spatial regularity for the flow map $\eta$, such as in \cite{alazard2014cauchy, kukavica2017local, luozhang2019MHD2.5}, is no longer required in this manuscript thanks to the presence of the surface tension on the boundary. 
\end{abstract}
\tableofcontents
\section{Introduction}
The goal of this manuscript is to investigate the solutions in Sobolev spaces for the following incompressible inviscid MHD equations in a moving domain with surface tension on the boundary:
\begin{equation}
\begin{cases}
\partial_t u+u\cdot\nabla u-B\cdot\nabla B+\nabla (p+\frac{1}{2}|B|^2)=0~~~& \text{in}~\DD; \\
\partial_t B+u\cdot\nabla B-B\cdot\nabla u=0,~~~&\text{in}~\DD; \\
\dive u=0,~~\dive B=0,~~~&\text{in}~\DD,
\end{cases}\label{MHD}
\end{equation}
describing the motion of conducting fluids in an electromagnetic field, where $\DD={\cup}_{0\leq t\leq T}\{t\}\times \Omega(t)$ and $\Omega(t)\subset \R^3$ is the domain occupied by the fluid  whose boundary $\p\Omega(t)$ moves with the velocity of the fluid.  Under this setting, the fluid velocity $u=(u_1,u_2,u_3)$, the magnetic field $B=(B_1,B_2,B_3)$, the fluid pressure $p$ and the domain $\DD$ are to be determined; in other words, given a simply connected bounded domain $\Omega(0)\subset \R^3$ and the initial data $u_0$ and $B_0$ satisfying the constraints $\di u_0=0$ and $\di B_0=0$, we want to find a set $\DD$ and the vector fields $u$ and $B$ solving \eqref{MHD} satisfying the initial conditions:
\begin{equation}
\Omega(0)=\{x: (0,x)\in \DD\},\q (u,B)=(u_0, B_0),\q \text{in}\,\,\{0\}\times \Omega_0.
\end{equation}
The quantity $P:=p+\frac{1}{2}|B|^2$ (i.e., the total pressure) plays an important role here in our analysis. It determines the acceleration of the moving surface boundary.  

We also require the following boundary conditions on the free boundary $\p\DD={\cup}_{0\leq t\leq T}\{t\}\times \p\Omega(t)$:
\begin{equation}
\begin{cases}
(\p_t+u\cdot \nabla )|_{\partial\DD}\in \mathcal{T}(\partial\DD) \\
P=\sigma\mathcal{H}~~~&\text{on}~\partial\DD, \\
B\cdot \mathcal{N}=0~~~&\text{on}~\partial\DD,
\end{cases}\label{MHDB}
\end{equation} 
where $\mathcal{N}$ is the exterior unit normal to $\p\Omega(t)$, $\sigma>0$ is the coefficient of surface tension, $\mathcal{H}$ is the mean curvature of the moving boundary embedded in $\mathbb{R}^3$. The first condition of \eqref{MHDB} means that the boundary moves with the velocity of the fluid. The second condition in \eqref{MHDB} suggests that the motion of the fluid is \textit{under the influence of the surface tension, as opposed to the case without surface tension} ($\sigma=0$).  Also, we remark here that $\mathcal{H}$ is a function of the unknowns and thus not known a priori. The third condition of \eqref{MHDB}, $B\cdot\mathcal{N}=0$ on $\p\Omega(t)$ implies that the fluid is a perfect conductor; in other words, the induced electric field $\mathcal{E}$ satisfies $\mathcal{E}\times \mathcal{N}=0$ on $\p\Omega(t)$. 

The physical energy is conserved, i.e., denoting $D_t:=\p_t+u\cdot \nab $, and invoking the divergence free condition for both $u$ and $B$, we have:

\[
\begin{aligned}
&~~~~~\frac{d}{dt}\left(\frac{1}{2}\int_{\Omega(t)} |u|^2 + \frac{1}{2}\int_{\Omega(t)} |B|^2+\sigma\int_{\p\Omega(t)} dS(\Omega(t))\right)\\
&=\int_{\Omega(t)} u\cdot D_tu+\int_{\Omega(t)}B\cdot D_t B+\frac{d}{dt}\left(\sigma\int_{\p\Omega(t)} dS(\Omega(t))\right)\\
&=-\int_{\Omega(t)}u\cdot \nab (p+\frac{1}{2}|B|^2)+\int_{\Omega(t)}u\cdot (B\cdot \nab B) +\int_{\Omega(t)}B\cdot (B\cdot \nab u)\\
&~~~~+\frac{d}{dt}\left(\sigma\int_{\p\Omega(t)} dS(\Omega(t))\right)\\
&=-\int_{\p\Omega(t)}(u\cdot \mathcal{N})P+\frac{d}{dt}\left(\sigma\int_{\p\Omega(t)} dS(\Omega(t))\right)\\
&~~~~+\int_{\Omega(t)}u\cdot (B\cdot \nab B)-\int_{\Omega(t)}u\cdot (B\cdot \nab B)=0,
\end{aligned}
\]where the first line in the last equality vanishes as shown in (2.1) in the paper \cite{shatah2008geometry}.

This motivates the construction of the higher order energy for \eqref{MHD}-\eqref{MHDB}. We refer Section \ref{tang} for the details. 

\subsection{History and background}\label{background}
The MHD equations describe the behavior of an electrically conducting fluid (e.g., a plasma) acted on by a magnetic field. In particular, the free-boundary MHD equations (also known as the plasma-interface problem) describe the phenomenon when the conducting fluid is separated from the outside wall by a vacuum. 
\bigskip

\noindent\textbf{An overview of the previous results}

In the absence of the magnetic field, i.e. $B=0$ in \eqref{MHD}, the problem reduces to the well-known incompressible free-boundary Euler equations, whose local well-posedness in Sobolev spaces was obtained first by Wu \cite{wu1997LWPww, wu1999LWPww}  for the irrotational case with $\sigma=0$, assuming the physical sign condition $-\nab_N p\geq \epsilon_0>0$ holds on $\p\DD_t$. This condition plays a crucial role in establishing the above well-posedness results for Euler equations. It was found by Ebin \cite{ebin1987equations} that the incompressible Euler equations is ill-posed when physical sign condition fails.   Extensions including the case without the irrotationalility assumption have been studied extensively in the past two decades, without attempting to be exhaustive, we refer \cite{christodoulou2000motion, coutand2007LWP, kukavica2017local, lindblad2002, lindblad2005well,   zhang2008free} for more details. 

On the other hand, the free boundary Euler equations behaves differently when $\sigma>0$. The surface tension is known to have regularizing effect on the moving surface. As a consequence, the physical sign condition is no longer needed when establishing the local well-posedness. We refer \cite{SchweizerFreeEuler, shatah2008geometry, shatah2008priori, shatah2011local} for more details. 

The free-boundary MHD equations, nevertheless, is far less well-understood than the free-boundary Euler equations. When $\sigma=0$, under the physical sign condition\footnotemark
\begin{equation}
-\nab_{\mathcal{N}} (p+\frac{1}{2}|B|^2) \geq \epsilon_0>0, \label{Tsign}
\end{equation}
 Hao-Luo \cite{hao2014priori} proved the a priori energy estimate with $H^4$ initial data and the local-wellposedness was established by Secchi-Trakhinin \cite{secchi2013well} and Gu-Wang \cite{gu2016construction}. Also, we remark here that in \cite{hao2018ill}, the authors proved that this problem is ill-posedness when \eqref{Tsign} is violated in the case of dimension 2.  We also mention here that in \cite{luozhang2019MHD2.5}, we proved a priori estimate with minimal regularity assumptions on the initial data (i.e., $v_0, B_0\in H^{2.5+\delta}$) in a \textit{small} fluid domain. Unlike the Euler equations, this assumption on the smallness of the volume of the fluid is crucial here due to the physical sign condition is unable to stabilize the MHD flow under \textit{low regularity assumptions}.  We will discuss more about this in the following paragraphs.

\footnotetext{In \cite{hao2017motion, sun2017well}, the authors studied the a priori energy estimate and local well-posedness, respectively, for the free-boundary MHD equations with nontrivial vacuum magnetic field under different stabilising assumptions. }
 
In \cite{GuoMHDSTviscous}, the authors proved the global well-posedness and exponentially decaying rate of the viscous and resistive free-boundary MHD equations when $\sigma>0$. In that paper, the kinematic viscosity and the magnetic diffusion allow them to control the enhanced regularity of the flow map, while this is impossible in the case of inviscid MHD without magnetic diffusion. In \cite{chendingMHDST}, Chen-Ding studied the vanishing viscosity-resistivity limit and the convergence rate for the viscous, resistive, free-boundary MHD system with surface tension in a half-space domain. Wang-Xin \cite{wangxinMHDST1} studied the global well-posedness of the incompressible resistive MHD system with surface tension near the equilibrium solution (constant strength magnetic field). However, to the best of our knowledge, \textit{NO result} that concerns the well-posedness theory of inviscid and non-resistive free-boundary MHD equations when $\sigma>0$ is available.  
\bigskip

\noindent\textbf{Our results}

The goal of this manuscript is to establish a priori energy estimates  for \eqref{MHD}-\eqref{MHDB} with fixed $\sigma>0$ when $u_0,B_0\in H^{3.5}(\Omega(0))$. Our result is an important first step to prove the local well-posedness for free-boundary MHD equations with surface tension, since the real conducting fluids have surface tension while the case without surface tension is just an idealized model. Moreover, we will show that the surface tension, in fact, has a stronger regularizing effect compare to that provided by the physical sign condition \eqref{Tsign}: We are able to get a better control of the normal component of the velocity field on the moving boundary through the boundary elliptic estimate due to the appearance of surface tension. As a result, our energy constructed in Chapter \ref{tang} contains at least two time derivatives, which removes the requirement of the flow map has to be $1/2$-derivatives more regular than $v$, $b$ in the case of no surface tension. We will give more illustration on this in Section \ref{section 1.3}.


\subsection{MHD system in Lagrangian coordinates and the main result}
We reformulate the MHD equations in Lagrangian coordinates, in which  the free domain becomes fixed. Let $\Omega$ be a bounded domain in $\R^3$. Denoting coordinates on $\Omega$ by $y=(y_1,y_2,y_3)$, we define $\eta:[0,T]\times \Omega\to\DD$ to be the flow map of the velocity $u$, i.e., 
\begin{equation}
\p_t \eta (t,y)=u(t,\eta(t,y)),\q
\eta(0,y)=y.
\end{equation}
We introduce the Lagrangian velocity, magnetic field and fluid pressure, respectively, by
\begin{equation}
v(t,y)=u(t,\eta(t,y)),\q
b(t,y)=B(t,\eta(t,y)),\q
q(t,y)=p(t,\eta(t,y)).
\end{equation}
Let $\p$ be the spatial derivative with respect to $y$ variable. We introduce the cofactor matrix $a=[\p\eta]^{-1}$, which is well-defined since $\eta(t,\cdot)$ is almost the identity map when $t$ is sufficiently small. It's worth noting that $a$ verifies the Piola's identity, i.e., 
\begin{equation}
\p_\mu a^{\mu\alpha} = 0.
\label{piola}
\end{equation}
Here, the Einstein summation convention is used for repeated upper and lower indices. In above and throughout, all Greek indices range over 1, 2, 3, and the Latin indices range over 1, 2. 

Denote the total pressure $P=p+\frac{1}{2}|B|^2$ and let $Q=P(t,\eta(t,y))$. Then \eqref{MHD}-\eqref{MHDB} can be reformulated as:
\begin{equation}
\begin{cases}
\partial_tv_{\alpha}-b_{\beta}a^{\mu\beta}\partial_{\mu}b_{\alpha}+a^{\mu}_{\alpha}\partial_{\mu}Q=0~~~& \text{in}~[0,T]\times\Omega;\\
\partial_t b_{\alpha}-b_{\beta}a^{\mu\beta}\partial_{\mu}v_{\alpha}=0~~~&\text{in}~[0,T]\times \Omega ;\\
a^{\mu\alpha}\partial_{\mu}v_{\alpha}=0,~~a^{\mu\alpha}\partial_{\mu}b_{\alpha}=0~~~&\text{in}~[0,T]\times\Omega;\\
v_3=0,~b_3=0~~~&\text{on}~\Gamma_0;\\
a^{\mu\alpha}N_{\mu}Q+\sigma(\sqrt{g}\Delta_g \eta^{\alpha})=0 ~~~&\text{on}~\Gamma;\\
a^{\mu\nu}b_{\nu}N_{\mu}=0 ~~~&\text{on}~\Gamma,\\
\end{cases}\label{MHDL}
\end{equation}
where $N$ is the unit outer normal vector to $\p\Omega$, $a^T$ is the transpose of $a$, $|\cdot|$ is the Euclidean norm and $\Delta_g$ is the Laplacian of the metric $g_{ij}$ induced on $\p\Omega(t)$ by the embedding $\eta$. Specifically, we have:
\begin{equation}\label{gij}
g_{ij}=\p_i\eta^{\mu}\p_j\eta_{\mu},~\Delta_g(\cdot)=\frac{1}{\sqrt{g}}\p_i(\sqrt{g}g^{ij}\p_j(\cdot)),\text{ where } g:=\det (g_{ij}).
\end{equation} For the details to derive the fifth equation of \eqref{MHDL} (the surface tension equation), we refer to Lemma 2.5 in \cite{disconzi2017prioriI} for readers.

For the sake of simplicity and clean notation, here we consider the model case when 
\begin{equation}
\Omega=\T^2\times (0,1),
\label{Omega}
\end{equation}
 where $ \partial\Omega=\Gamma_0\cup\Gamma$ and $\Gamma=\T^2\times \{1\}$ is the top (moving) boundary, $\Gamma_0=\T^2\times\{0\}$ is the fixed bottom. Using a partition of unity, e.g., \cite{DKT}, a general domain can also be treated with the same tools we shall present. However, choosing $\Omega$ as above allows us to focus on the real issues of the problem without being distracted by the cumbersomeness of the partition of unity. Let $N$ stands for the outward unit normal of $\p\Omega$. In particular, we have $N=(0,0,-1)$ on $\Gamma_0$ and $N=(0,0,1)$ on $\Gamma$. 

In this paper, we prove:

\begin{thm}\label{MHDthm}
Let $\Omega$ be defined as in \eqref{Omega}. Assume that $v_0\in H^{3.5}(\Omega)\cap H^4(\Gamma)$ and $b_0\in H^{3.5}(\Omega)$ be divergence free vector fields with $b_0\cdot N=0$ on $\p\Omega$. Assume $(\eta, v, b, Q)$ to be any solution of \eqref{MHDL} with initial data $v_0$ and $b_0$. Define
\begin{equation}\label{N(t)}
\begin{aligned}
N(t)=\|\eta\|_{3.5}^2&+\|v\|_{3.5}^2+\|v_t\|_{2.5}^2+\|v_{tt}\|_{1.5}^2+\|v_{ttt}\|_{0}^2+\|b\|_{3.5}^2+\|b_t\|_{2.5}^2+\|b_{tt}\|_{1.5}^2+\|b_{ttt}\|_{0}^2\\
&+\|Q\|_{3.5}^2+\|Q_t\|_{2.5}^2+\|Q_{tt}\|_1^2.
\end{aligned}
\end{equation}
Then there exists a $T>0$, chosen sufficiently small, such that $N(t)\leq C_0$ for all $t\in [0,T]$, where $C_0$ only depends on $\|v_0\|_{3.5},\|b_0\|_{3.5}, \|v_0\|_{4,\Gamma}$. Here we denote $\|f\|_{s}:= \|f(t,\cdot)\|_{H^s(\Omega)}$ for any function $f(t,y)\text{ on }[0,T]\times\Omega$, and $\|f\|_{s,\Gamma}:= \|f(t,\cdot)\|_{H^s(\Gamma)}$ for any $f(t,y)\text{ on }[0,T]\times\Gamma$.
\end{thm}

\subsection{Strategy and organisation of the paper}\label{section 1.3}

\paragraph*{Notations.} All definitions and notations will be defined as they are introduced. In addition, a list of symbols will be given at the end of this section for a quick reference.
\defn The $L^2$-  based Sobolev spaces are denoted by $H^s(\Omega)$, where
we abbreviate corresponding norm $\|\cdot\|_{H^r(\Omega)}$ as $\|\cdot\|_{r}$ when no confusion can arise. We denote by $H^s(\Gamma)$ the Sobolev space of functions defined on the boundary $\Gamma$ ($\Gamma$=$\Gamma_0$ or $\Gamma$), with norm $\|\cdot\|_{s,\Gamma}$. 

\nota We use $\epsilon$ to denote a small positive constant which may vary from expression to expression. Typically, $\epsilon$ comes from choosing sufficiently small time, from Lemma \ref{estimatesofa} and from the Young's inequality. 
\nota We use $P=P(\cdots)$ to denote a generic polynomial in its arguments.\\

\noindent\textbf{Gronwall-Type argument and div-curl estimates}

To derive the a priori estimates in Theorem \ref{MHDthm}, we need to do the div-curl-boundary decomposition for $v$, $b$ and their time derivatives and finally need a Gronwall-type control
\[
N(t)\lesssim P(N(0))+P(N(t))\int_0^t P(N(s))ds
\] holds in some time interval $[0,T]$. This implies $N(t)\lesssim C(N(0))$ for some constant only depending on $N(0)$, i.e., the initial data.

The divergence control is easy thanks to the Eulerian divergence-free condition for $v$ and $b$. The vorticity control will be derived from the evolution equations of the Eulerian vorticity of $v$ and $b$, as shown in \cite{luozhang2019MHD2.5}. Its computation requires the repeated use of Kato-Ponce inequalities in Lemma \ref{KatoPonce}. 

The boundary terms of $v$ and $b$ are treated in different ways: We can control the normal component of $v$ and its time derivates on the boundary thanks to the boundary elliptic estimates and the comparison between the normal component $X^3$ and that of the tangential projection $\Pi X$ as shown in Section \ref{bdryofv}. For the boundary control of $b$, we invoke the identity $b=(b_0\cdot\p)\eta$ (see Lemma \ref{GW}) and the condition $b_0\cdot N=0$ on the boundary to see that $b=(b_0\cdot \TP)\eta$ on the boundary $\Gamma$. Then applying again the boundary elliptic estimates gives the control of $\TP\eta$. The time derivative counterparts can be controlled in the same way.

\bigskip

\noindent\textbf{Boundary estimates of the velocity}

From above, we need to control $v^3,v_t^3$ and $v_{tt}^3$. One can differentiate the surface tension equation in time variable to derive an ellptic equation for $v$ 
\begin{equation}
\Delta_g v^3=-\frac{1}{\sigma}a^{\mu 3}N_{\mu}Q_t+\cdots \label{elliptic rough}
\end{equation}
and then apply the elliptic estimates to control $v$. However this is not valid for higher order time derivative since we do not have enough regularity for $Q_{tt}$ or $Q_{ttt}$. To solve this problem, we use the method in \cite{disconzi2017prioriI}: Let $X$ be a vector field. Then one can compare the $X\cdot N=X^3$ with the normal projection $\Pi X$ as in Lemma \ref{tgproj}. Specifically, this is based on a simple fact that 
\[
X^3-(\Pi X)^3=g^{kl}\p_k\eta^3\p_l\eta_{\lambda}X^{\lambda},
\]
where the error term on the RHS can be controlled in a routine fashion. 

The interior estimates of $\TP v_{tt}$ and $\TP^2 v_t$, together with that of $v_{ttt}$ and $b_{ttt}$, are derived in the tangential energy estimates. To see this for $\TP v_{tt}$, one can first compute $\|v_{ttt}\|_0^2+\|b_{ttt}\|_0^2$, where all the terms containing $b$ with highest order actually vanish due to remarkable cancellation as shown in \eqref{tgtttcancel} and \eqref{tgvttcancel}. The main term is a boundary integral containing full derivative of $Q$ (i.e. $Q_{ttt}$) after integration by parts
\[
-\int_0^t\int_{\Gamma} a^{\mu\alpha}Q_{ttt}\p_t^3 v_{\alpha} N_{\mu}=-\int_0^t\int_{\Gamma} \p_t^3(\underbrace{a^{\mu\alpha}Q N_{\mu}}_{=-\sigma\sqrt{g}\Delta_g \eta^{\alpha}})\p_t^3 v_{\alpha} +\cdots
\]

Invoking the identities
\[
\p_t(\sqrt{g}\Delta_g \eta^{\nu})=\p_i\left(\sqrt{g}g^{ij}(\delta^{\alpha}_{\lambda}-g^{kl}\p_k\eta^{\alpha}\p_l\eta^{\lambda})\p_j v^{\lambda})+\sqrt{g}(g^{ij}g^{kl}-g^{lj}g^{ik})\p_j\eta^{\alpha}\p_k\eta^{\lambda}\p_l v^{\lambda}\right)
\] 
$$\Pi_{\lambda}^{\alpha}=\delta_{\lambda}^{\alpha}-g^{kl}\p_k\eta^{\alpha}\p_l\eta_{\lambda},$$ 
 we can get a coercive term
\[
-\frac{1}{2}\int_{\Gamma}\sqrt{g}g^{ij}\p_i(\Pi^{\alpha}_{\mu}\p_t^2v_{\alpha})\p_j(\Pi^{\mu}_{\lambda}\p_t^2v^{\lambda})
\] after integrating by parts. This term is almost equal to $-\frac{1}{2}\|\TP v_{tt}\|_{0,\Gamma}^2$ since $\sqrt{g}g^{ij}\sim \delta^{ij}$ within short time interval. Analogous computation also holds for $\TP^2v_t$. This concludes the boundary estimates for the velocity. 
\bigskip

\noindent\textbf{Surface tension stabilizes the flow}

As stated in Section \ref{background}, the physical sign condition \eqref{Tsign} is \textit{insufficient to regularize} the motion of free-boundary conducting liquid in low regularity Sobolev spaces (i.e., whenever $\p^2\eta\notin L^\infty$) and extra regularity assumptions are required (e.g., the smallness of the fluid volume). In this manuscript, we show that the presence of the surface tension provides stronger regularizing effect for free-boundary MHD equations. This is due to that one can time-differentiate the boundary condition to derive an elliptic equation for $v$ \eqref{elliptic rough},  which allows us to control the normal component of $v$ on the moving boundary via elliptic estimates. As a consequence, this helps us avoid controlling the full spatial derivatives of $v$ and thus the extra regularity assumptions on $\eta$ is no longer needed. 
\bigskip 

\noindent\textbf{Illustration on the regularity requirement of the flow map}

To understand how the requirement of the regularity of the flow map $\eta$ appears, one first need to realize a crucial difference between Euler's equations and MHD equations: There is NO analogue of ``\textit{irrotationality assumption}" for a conducting fluid due to the presence of the Lorentzian force term $B\cdot \nab B$. Physically, this is due to that the Lorentzian force twists the trajectory of an electric particle in a magnetic field and produces vorticity even if the initial data is curl-free. Mathematically, as shown in our previous work \cite{luozhang2019MHD2.5}, the well-known Cauchy invariance fails. The Cauchy invariance, however, is required to control the flow map if it is $1/2$-derivative more regular than the velocity, and so one has to introduce the smallness assumption on the fluid domain to compensate the failure of the Cauchy invariance in the case of no surface tension. 

We mention here that in \cite{gu2016construction}, the authors adapted Alinhac's good unknowns to remove the extra regularity on the flow map $\eta$ in the case of no surface tension. However, this requires the initial data to be in $H^4$.

\bigskip

\noindent\textbf{Effects brought by the surface tension}
  
As stated above, the surface tension helps us to avoid controlling the higher regularity of $\eta$ owing to the boundary elliptic estimates. It is natural to ask if the surface tension makes any negative effect on controlling other quantities.
We point out that the contribution of the surface tension here is the stronger boundary control of $v$, which requires higher regularity of $\eta$ in the case of no surface tension as shown in our previous work \cite{luozhang2019MHD2.5}. On the other hand, the side effects on the control of the magnetic field caused by the surface tension are mainly technical difficulties, as shown in Chapter \ref{tang}. The strong coupling between the velocity and magnetic fields does not worsen those technical terms. Therefore, one can see the surface tension contributes mainly in the kinetic part, while the impact on controlling other quantities can be controlled.

\bigskip

\noindent\textbf{Elliptic Estimates of Pressure}

Our computation above produces some term like $\epsilon\|Q_{tt}\|_1^2$ after using Young's inequality. Therefore we need to do the pressure estimates, which can be derived from the elliptic equations of the pressure. $Q$ and $Q_t$ can be straightforward controlled by using the standard elliptic estimates $\|u\|_{s}\lesssim \|f\|_{s-2}+\|g\|_{s-1.5,\p\Omega}+\|u\|_{0}(\forall s\geq 2)$. However for the $H^1$-control of $Q_{tt}$, we need a low regularity estimate as in Lemma \ref{H1elliptic} proved in \cite{ignatova2016local}, and the trace lemma with negative Sobolev index in Lemma \ref{divtrace}. Finally, one needs to re-write these estimates in terms of the sum of initial data and time integral of the quantities in $N(t)$ to finish the Gronwall-type control of $N(t)$ as above.

\paragraph*{List of symbols:}
\begin{itemize}
\item $\epsilon$:  A small positive constant which may vary from expression to expression.
\item $a=[\p \eta]^{-1}$: The cofactor matrix;
\item $\|\cdot\|_{s}$: We denote $\|f\|_{s}:= \|f(t,\cdot)\|_{H^s(\Omega)}$ for any function $f(t,y)\text{ on }[0,T]\times\Omega$.
\item $\|\cdot\|_{s,\Gamma}$: We denote $\|f\|_{s,\Gamma}:= \|f(t,\cdot)\|_{H^s(\Gamma)}$ for any function $f(t,y)\text{ on }[0,T]\times\Gamma$.
\item $P(\cdots)$: A generic polynomial in its arguments;
\item $\PP$: $\PP=P(\|v\|_{3.5},\|v_t\|_{2.5},\|v_{tt}\|_{1.5},\|v_{ttt}\|_{0},\|b\|_{3.5},\|b_t\|_{2.5},\|b_{tt}\|_{1.5},\|b_{ttt}\|_{0})$;
\item $N(t)$: 
\[
\begin{aligned}
N(t)=\|\eta\|_{3.5}^2&+\|v\|_{3.5}^2+\|v_t\|_{2.5}^2+\|v_{tt}\|_{1.5}^2+\|v_{ttt}\|_{0}^2+\|b\|_{3.5}^2+\|b_t\|_{2.5}^2+\|b_{tt}\|_{1.5}^2+\|b_{ttt}\|_{0}^2\\
&+\|Q\|_{3.5}^2+\|Q_t\|_{2.5}^2+\|Q_{tt}\|_1^2;
\end{aligned}
\]
\item $\TP=\p_1,\p_2$: Tangential differential operators. 
\end{itemize}

\section{Preliminary Lemmas}
\label{section 2}

The first lemma records some basic estimates of the cofactor matrix $a$, which shall be used throughout the rest of the manuscript. 

\begin{lem}\label{estimatesofa}
Suppose $\| v\|_{L^{\infty}([0,T];H^{3.5}(\Omega))}\leq M$. If $T\leq \frac{1}{CM}$ for a sufficiently large constant $C$, then the following estimates hold:

(1) $\|\eta\|_{3.5}\leq C$ for $t\in [0,T]$;

(2) $\det(\p\eta(t,x))=1$ for $(x,t)\in\Omega\times [0,T]$;

(3) $\|a(\cdot,t)\|_{H^{2.5}}\leq C$ for $t\in [0,T]$;

(4) $\|a_t(\cdot,t)\|_{r}\leq C\|\p v\|_{r}$ for $t\in [0,T]$, $0\leq r\leq 2.5$;

(4)' $\|a_t(\cdot,t)\|_{L^p}\leq C\|\p v\|_{L^p}$ for $t\in [0,T]$, $1\leq p\leq\infty$;

(5) $\|a_{tt}(\cdot,t)\|_{r}\leq C\|v\|_{2.5+\delta}\|\p v\|_{r}+C\|\p v_t\|_{r}$, for $t\in [0,T]$, $0< r\leq 1.5$;

(6) $\|a_{ttt}(\cdot,t)\|_{r}\leq C\|\p v\|_{r}\|v\|_{2.5+\delta}^2+C\|\p v_t\|_{r}\|v\|_{2.5+\delta}+C\|\p v_{tt}\|_{r}$, for $t\in [0,T]$, $0< r\leq 0.5$;

(6)' $\|a_{ttt}(\cdot,t)\|_{L^p}\leq C\|\p v\|_{L^p}\|v\|_{2.5+\delta}^2+C\|\p v_t\|_{L^p}\|v\|_{2.5+\delta}+C\|\p v_{tt}\|_{L^p}$, for $t\in [0,T]$, $1\leq p\leq \infty$;

(7) For every $0<\epsilon\leq 1$, there exists a constant $C>0$ such that for all $0\leq t\leq T':=\min\{\frac{\epsilon}{CM},T\}>0$, we have $$\|a^{\mu}_{\nu}-\delta^{\mu}_{\nu}\|_{2.5}\leq\epsilon, ~~\|a^{\mu\alpha}a^{\nu}_{\alpha}-\delta^{\mu\nu}\|_{2.5}\leq \epsilon.$$ In particular $a^{\mu}_{\alpha}a^{\nu}_{\alpha}$ satisfies the ellpticity condition $$a^{\mu\alpha}a^{\nu}_{\alpha}\xi_{\mu}\xi_{\nu}\geq \frac{1}{C}|\xi|^2~~~~\forall \xi\in \mathbb{R}^3;$$

(8) $D a_{\alpha}^\mu=-a^{\mu}_{\nu}\p_{\beta}D \eta^\nu a^{\beta}_{\alpha}$, for $D=\p,\p_t.$
\end{lem}
\begin{proof}
(1)-(7) is Lemma 3.1 in that paper. (8) is derived from differentiating the identity $a=a:\p\eta :a$. We refer \cite{ignatova2016local} for the details. 
\end{proof}

The next lemma is to introduce a low regularity elliptic estimates, and we refer Lemma 3.2 in \cite{ignatova2016local} for the proof. It will be used to control $\|Q_{tt}\|_{1}$.
\begin{lem}\label{H1elliptic}
Assume $A^{\mu\nu}$ satisfies $\|A\|_{L^{\infty}}\leq K$ and the ellipticity $A^{\mu\nu}(x)\xi_{\mu}\xi_{\nu}\geq \frac{1}{K}|\xi|^2$ for all $x\in\Omega$ and $\xi\in\mathbb{R}^3$. Assume $W$ to be an $H^1$ solution to 
\begin{equation}\label{lowelliptic}
\begin{cases}
\p_{\nu}(A^{\mu\nu}\p_{\mu}W)=\dive \pi &~in~\Omega \\
A^{\mu\nu}\p_{\nu}WN_{\mu}=h &~on~\p\Omega,
\end{cases}
\end{equation} where $\pi,\dive \pi \in L^2(\Omega)$ and $h\in H^{-0.5}(\p\Omega)$ with the compatibility condition $$\int_{\p\Omega}(\pi\cdot N-h)dS=0.$$ If $\|A-I\|_{L^{\infty}}\leq\epsilon_0$ which is a sufficently small constant depending on $K$, then we have:
\begin{equation}
\|W-\overline{W}\|_{1}\lesssim\|\pi\|_{0}+\|h-\pi\cdot N\|_{-0.5,\p\Omega},\text{ where }\overline{W}:=\frac{1}{|\Omega|}\int_{\Omega} W dy,
\end{equation}
and 
\begin{equation}
\|W\|_{1}\lesssim\|\pi\|_{0}+\|h-\pi\cdot N\|_{-0.5,\p\Omega}+\|W\|_{0,\Gamma}.
\end{equation}
\end{lem}

Furthermore, we need the regularity estimate for the flow map $\eta$ on the boundary. $\eta$ verifies an elliptic equation on $\Gamma$ which yields a gain of regularity.  It has been pointed out in \cite{shatah2008geometry} that, this regularity gain is geometric in nature and has nothing to do with the interior regularity (see the counterexamples in \cite{shatah2008geometry}). We will need $H^4(\Gamma)$ estimate of $\eta$ in this paper and we point out that this estimate can be upgraded to $H^5(\Gamma)$.
\begin{prop}\label{dhj}
We have the estimate 
\begin{equation}\label{etabdry}
\|\eta\|_{4,\Gamma}\leq P(\|Q\|_{2, \Gamma}).
\end{equation}
\end{prop}
\begin{proof}
The proof is based on the conclusion in Dong-Kim \cite{donghongjieBMO}: It suffices to verify the coefficient is bounded in BMO semi-norm. The detailed computation is almost the same as in Proposition 3.4 in \cite{disconzi2017prioriI} so we omit it.
\end{proof}

The next lemma is to introduce the identities about the magnetic field $b$. It was first discovered by Wang in \cite{wangyanjin2012} and used on the free-boundary MHD equations by Gu-Wang in \cite{gu2016construction}. This lemma reveals the regularising effect of the magnetic field $b$; in particular, the flow map $\eta$ is more regular in the direction of $b_0$.

\begin{lem}\label{GW}
Let $(v,b,\eta)$ be a solution to \eqref{MHDL} with initial data $(v_0,b_0,\eta_0)$. Then the following two identities hold:
\begin{equation}\label{GW1}
a^{\nu\alpha} b_{\alpha}=b_0^{\nu},
\end{equation} 
\begin{equation}\label{GW2}
b^{\beta}=(b_0\cdot\p)\eta^{\beta}=b_0^{\nu}\p_{\nu}\eta^{\beta}.
\end{equation} 
\end{lem}
\begin{proof}
For \eqref{GW1}, we multiply $a^{\nu\alpha}$ to the second equation of \eqref{MHDL} to get $$a^{\nu\alpha}\p_tb_{\alpha}=a^{\nu\alpha}b_{\beta}a^{\mu\beta}\p_{\mu}\p_t\eta_{\alpha}=a^{\nu\alpha}b_{\beta}\p_t(\underbrace{a^{\mu\beta}\p_{\mu}\eta_{\alpha}}_{=\delta^{\beta}_{\alpha}})-b_{\beta}\p_ta^{\mu\beta}(\underbrace{\p_{\mu}\eta_{\alpha}a^{\nu\alpha}}_{\delta_{\mu}^{\nu}})=-b_{\alpha}\p_ta^{\nu\alpha},$$ so $\p_t(a^{\nu\alpha}b_{\alpha})=0$ and thus $a^{\nu\alpha}b_{\alpha}=b_0^{\nu}$. For \eqref{GW2}, it can be easily derived by multiplying $\p_{\nu}\eta_{\beta}$ on the both sides of \eqref{GW1} and using $a:\p\eta =I$.
\end{proof}

The last three lemmas record the results of basic PDE theory. The first one is the well-known Kato-Ponce commutator estimates, the proof of which can be found in \cite{kato1988commutator} and \cite{lidong2019commutator}.
\begin{lem}\label{KatoPonce}
Let $J=(I-\Delta)^{1/2}$, $s\geq 0$. Then the following estimates hold:

(1) $\forall s\geq 0$, we have 
\begin{equation}\label{product}
\|J^s(fg)\|_{L^2}\lesssim \|f\|_{W^{s,p_1}}\|g\|_{L^{p_2}}+\|f\|_{L^{q_1}}\|g\|_{W^{s,q_2}},
\end{equation}with $1/2=1/p_1+1/p_2=1/q_1+1/q_2$ and $2\leq p_1,q_2<\infty$;

(2) $\forall s\in (0,1)$, we have 
\begin{equation}\label{kato1}
\|J^s(fg)-f(J^s g)-(J^s f)g\|_{L^p}\lesssim \| f\|_{W^{s_1,p_1}}\|g\|_{W^{s-s_1,p_2}},
\end{equation}
where $0<s_1<s$ and $1/p_1+1/p_2=1/p$ with $1<p<p_1,p_2<\infty$;

(2') $\forall s\geq 1$, we have
\begin{equation}\label{kato3}
\|J^s(fg)-(J^sf)g-f(J^sg)\|_{L^p}\lesssim\|f\|_{W^{1,p_1}}\|g\|_{W^{s-1,q_2}}+\|f\|_{W^{s-1,q_1}}\|g\|_{W^{1,q_2}}
\end{equation} for all the $1<p<p_1,p_2,q_1,q_2<\infty$ with $1/p_1+1/p_2=1/q_1+1/q_2=1/p$.

(3) $\forall s\geq 1$, we have
\begin{equation}\label{kato2}
\|J^s(fg)-f(J^s g)\|_{L^2}\lesssim \|f\|_{W^{s,p_1}}\|g\|_{L^{p_2}}+\|f\|_{W^{1,q_1}}\|g\|_{W^{s-1,q_2}},
\end{equation}
where $1/2=1/p_1+1/q_1=1/p_2+1/q_2$ with $1<p<p_1,p_2<\infty$;

(3') $\forall s\geq 0$ and $1<p<\infty$, we have
\begin{equation}\label{KATO}
\|J^s(fg)-f(J^s g)\|_{L^p}\lesssim \|\partial f\|_{L^{\infty}}\|J^{s-1} g\|_{L^p}+\|J^s f\|_{L^p}\|g\|_{L^{\infty}};
\end{equation}

(3'') For $1<p<\infty$ and $1<p_1,q_1,p_2,q_2\leq\infty$ satisfying $1/p=1/p_1+1/p_2=1/q_1+1/q_2$, the following hold:
\begin{itemize}
\item If $0<s\leq 1$, then 
\begin{equation}\label{lidong1}
\|J^s(fg)-f(J^s g)\|_{L^p}\lesssim \|J^{s-1} \p f\|_{L^{p_1}}\|g\|_{L^{p_2}};
\end{equation}
\item If $s>1$, then 
\begin{equation}\label{lidong2}
\|J^s(fg)-f(J^s g)\|_{L^p}\lesssim \|J^{s-1}\p f\|_{L^{p_1}}\|g\|_{L^{p_2}}+\|\p f\|_{L^{q_1}}\|J^{s-2} \p g\|_{L^{q_2}}.
\end{equation}
\end{itemize}
\end{lem}

\begin{flushright}
$\square$
\end{flushright}

The second lemma is a refined version of the Sobolev interpolation proved in \cite{sobolevinterpolation}. It will be used to estimate  the lower order error terms. 
\begin{lem}\label{gnsineq}
Suppose $\Omega$ is a domain in $\mathbb{R}^d$. Suppose also $0\leq s_1\leq s\leq s_2$ and $1\leq p,p_1,p_2\leq\infty$. If the condition
\[
1\leq s_2\in\mathbb{Z}\text{ and }p_2=1\text{ and }s_2-s_1\leq 1-\frac{1}{p_1}
\] \textit{fails}, then the following interpolation result holds for all $\theta\in (0,1)$:
\[
\|f\|_{W^{s,p}(\Omega)}\lesssim_{d,s_1,s_2,p_1,p_2,\Omega,\theta}\|f\|_{W^{s_1,p_1}(\Omega)}^{\theta}\|f\|_{W^{s_2,p_2}(\Omega)}^{1-\theta},
\]provided $s=\theta s_1+(1-\theta)s_2$ and $1/p=\theta /p_1 +(1-\theta)/p_2$ hold.
\end{lem}
\begin{flushright}
$\square$
\end{flushright}
The last basic lemma is a Sobolev trace-type lemma which allows us to use trace theorem for the Sobolev spaces with negative order in some special cases. It can be found in Theorem A.2.4 in \cite{H-0.5trace} on page 251.
\begin{lem}\label{divtrace}
For $1<p<\infty$, we define the function space for vector fields $X\in\mathbb{R}^d$: $$L_{\dive }^p(\Omega):=\{X\in L^p(\Omega):\dive X\in L^p(\Omega)\}$$ with the graph norm $$\|X\|_{\dive}:=(\|X\|_{L^p(\Omega)}^p+\|\dive X\|_{L^p(\Omega)}^p)^{1/p}.$$ Then there is a unique continuous linear operator $$Tr_N:L_{\dive }^p(\Omega)\to W^{-1/p,p}(\p\Omega)$$ such that $Tr_NX=(X\cdot N)|_{\p\Omega}$ for each $X\in C(\bar{\Omega})\cap L_{\dive}^p(\Omega).$
\end{lem}

\begin{flushright}
$\square$
\end{flushright}

\section{Pressure Estimates}\label{ellipticQ}
In this section we prove the following bounds for $Q$, $Q_t$ and $Q_{tt}$, which will be repeatedly used in the following chapters. Our conclusion is the following proposition. 
\begin{prop}\label{QQ}
Assume Lemma \ref{estimatesofa} holds. Then the total pressure $Q$ satisfies:
\begin{equation}\label{estimatesofQ}
\|Q\|_{3.5}\lesssim \|v\|_{2.5+\delta}^2+\|b\|_{2.5+\delta}^2+\|v_t\|_{2.5}+\|b_0\|_{2.5}\|b\|_{3.5}+1+c\lesssim \PP;
\end{equation}
\begin{equation}\label{estimatesofQt}
\begin{aligned}
\|Q_t\|_{2.5}&\lesssim \|v\|_{2.5+\delta}(\|Q\|_{2.5}+\|v_t\|_{1.5})+\|v\|_2(\|v\|_{2.5+\delta}^2+\|v_t\|_{2.5})+\|v_{tt}\|_{1.5} \\
&~~~~+\|b_0\|_3\|b_t\|_{1.5}+\|b\|_{2.5+\delta}\|b_t\|_{1.5}+\|v\|_{2.5+\delta}\|b_0\|_{2.5}\|b\|_2+\|b_0\|_2\|b_t\|_{2.5}\\
&\lesssim \PP ;
\end{aligned}
\end{equation}
\begin{equation}\label{estimatesofQtt}
\begin{aligned}
\|Q_{tt}\|_{1}&\lesssim (\|v_t\|_1+\|\p Q\|_1)(\|v\|_{1.5}\|v\|_{2.5+\delta}+\|v_t\|_{1.5})+\|v\|_{2.5+\delta}(\|Q_t\|_1+\|v_{tt}\|_{0})+\|v_{ttt}\|_0\\
&~~~~+\|v\|_{2.5+\delta}(\|v\|_{2.5+\delta}^2 \|v\|_1+\|v\|_{2.5+\delta}\|v_t\|_1+\|v_{tt}\|_1)\\
&~~~~+\|v_t\|_{2.5}+\|v\|_2\|Q_t\|_1+(\|v_t\|_{1.5}+\|v\|_2^2)\|Q\|_{2.5}+\|v\|_{2.5+\delta}\|v\|_{2.5}\\
&~~~~+\|v_t\|_1\|b_0\|_2\|b\|_{2.5+\delta}+\|v\|_{1.5}\|b_t\|_{1.5}+\|b_0\|_1\|b_{tt}\|_{1.5}\\
&\lesssim \PP,
\end{aligned}
\end{equation}where $\delta>0$ is a constant to be determined later, and can be sufficiently small if needed. $\PP$ denotes $P(\|v\|_{3.5},\|v_t\|_{2.5},\|v_{tt}\|_{1.5},\|v_{ttt}\|_{0},\|b\|_{3.5},\|b_t\|_{2.5},\|b_{tt}\|_{1.5},\|b_{ttt}\|_{0})$ throughout this paper as shown in the list of notations.
\end{prop}

As for the basic idea of the proof, the control of $Q$ and $Q_t$ will be derived from the standard elliptic estimates, whereas the control of $Q_{tt}$ needs Lemma \ref{H1elliptic} and Lemma \ref{divtrace} due to the low regularity.

\subsection{Control of $Q$ and $Q_t$: Standard elliptic estimate}
First the total pressure $Q$ verifies an elliptic equation as computed in Section 3 of \cite{luozhang2019MHD2.5}: In $\Omega$, we have
\begin{equation}\label{Qin}
\p^{\mu}\p_{\mu}Q=\p_ta^{\nu\alpha}\p_{\nu}  v_{\alpha}+\p_{\nu}((\delta^{\mu\nu}-a^{\mu}_{\alpha}a^{\nu\alpha})\p_{\mu}Q)+a^{\nu\alpha}\p_{\nu}b_0^{\mu}\p_{\mu}b_{\alpha}+\p_{\beta}b_{\gamma}a^{\nu\gamma}a^{\beta\alpha}\p_{\nu}b_{\alpha}-\p_{\beta}b_0^{\mu}a^{\beta\alpha}\p_{\mu}b_{\alpha}.
\end{equation} 
The boundary condition of $Q$ can be derived by contracting the first equation of \eqref{MHDL} with $a^{\mu\alpha}N_{\mu}=a^{3\alpha}$ and then restricting to the boundary:
\begin{equation}\label{Qbd}
\frac{\p Q}{\p N}=(\delta^{\mu 3}-a^{\mu}_{\alpha}a^{3\alpha})\p_{\mu}Q-a^{3\alpha}\p_t v_{\alpha}+a^{3\alpha}b_0^{\nu}\p_{\nu}b_{\alpha}.
\end{equation}

Denoting the RHS of \eqref{Qin} and \eqref{Qbd} by $f$ and $g$ and invoking the standard elliptic estimate, we obtain $$\|Q\|_{3.5}\lesssim\|f\|_{1.5}+\|g\|_{2,\Gamma}+\|Q\|_{0,\Gamma},$$ where the last term can be controlled by using the boundary condition. We apply the multiplicative Sobolev inequality (as a corollary of Kato-Ponce product estimate \eqref{product}) to get the following control of $f$ and $g$:
\begin{equation}\label{Qinterior}
\begin{aligned}
\|f\|_{1.5}&\leq\|\p_ta^{\nu\alpha}\p_{\nu}v_{\alpha}\|_{1.5}+\|(\delta^{\mu\nu}-a^{\mu}_{\alpha}a^{\nu\alpha})\p_{\mu}Q\|_{2.5} \\
&~~~~+\|a^{\nu\alpha}\p_{\nu}b_0^{\mu}\p_{\mu}b_{\alpha}\|_{1.5}+\|\p_{\beta}b_{\gamma}a^{\nu\gamma}a^{\beta\alpha}\p_{\nu}b_{\alpha}\|_{1.5}+\|\p_{\beta}b_0^{\mu}a^{\beta\alpha}\p_{\mu}b_{\alpha}\|_{1.5} \\
&\lesssim \|a_t\|_2\|v\|_{2.5+\delta}+\epsilon\|Q\|_{3.5}+\|b\|_{2.5+\delta}\|b_0\|_{2.5+\delta} \\
&\lesssim \|v\|_{2.5+\delta}^2+\|b\|_{2.5+\delta}^2+\epsilon\|Q\|_{3.5};
\end{aligned}
\end{equation}

\begin{equation}\label{Qbdry}
\|g\|_{2,\Gamma}\lesssim\epsilon\|Q\|_{3.5}+\|v_t\|_{2.5}+\|b_0\|_{2.5}\|b\|_{3.5},
\end{equation} where we use trace lemma to control $g$.
It remains to bound $\|Q\|_{0,\Gamma}$. Invoking the surface tension equation, i.e., the fifth equation in \eqref{MHDL} and Lemma \ref{estimatesofa}, we have $$\|Q\|_{0,\Gamma}\lesssim \|\sqrt{g}\Delta_g\eta\|_{0,\Gamma}\lesssim 1$$
Therefore, after absorbing the $\epsilon$-term to LHS, one has 
\begin{equation}\label{Q}
\|Q\|_{3.5}\lesssim \|v\|_3^2+\|b\|_3^2+\|v_t\|_{2.5}+\|b_0\|_{2.5}\|b\|_{3.5}+1.
\end{equation}

We next estimate $Q_t$ in $H^{2.5}$.  Taking time derivative in \eqref{Qin} and \eqref{Qbd}, we get the following elliptic equation for $Q_t$
\begin{equation}\label{Qtin}
\begin{aligned}
\p^{\mu}\p_{\mu}Q_t&=\p_{tt}a^{\nu\alpha}\p_{\nu}v_{\alpha}+\p_ta^{\nu\alpha}\p_{\nu}\p_t v_{\alpha}\\
&~~~~-\p_{\nu}(\p_ta^{\mu}_{\alpha}a^{\nu\alpha}\p_{\mu}Q)-\p_{\nu}(a^{\mu}_{\alpha}\p_t a^{\nu\alpha}\p_{\mu}Q)+\p_{\nu}((\delta^{\mu\nu}-a^{\mu}_{\alpha}a^{\nu}_{\alpha})\p_{\mu}Q_t) \\
&~~~~+a^{\nu\alpha}_t\p_{\nu}b_0^{\mu}\p_{\mu}b_{\alpha}+a^{\nu\alpha}\p_{\nu}b_0^{\mu}\p_t\p_{\mu}b_{\alpha}+\p_t(\p_{\beta}b_{\gamma}\p_{\nu}b_{\alpha})a^{\nu\gamma}a^{\beta\alpha}+\p_{\beta}b_{\gamma}\p_t(a^{\nu\gamma}a^{\beta\alpha})\p_{\nu}b_{\alpha}\\
&~~~~-\p_{\beta}b_0^{\mu}a^{\beta\alpha}\p_t\p_{\mu}b_{\alpha}-\p_{\beta}b_0^{\mu}a^{\beta\alpha}_t\p_{\mu}b_{\alpha}\\
&=:f^*
\end{aligned}
\end{equation} with the boundary condition
\begin{equation}\label{Qtbd}
\begin{aligned}
\frac{\p Q_t}{\p N}&=(\delta^{\mu 3}-a^{\mu}_{\alpha}a^{3\alpha})\p_{\mu}Q_t-\p_t(a^{\mu}_{\alpha}a^{3\alpha})\p_{\mu}Q-a^{3\alpha}\p_{tt} v_{\alpha}+a^{3\alpha}b_0^{\nu}\p_{\nu}\p_t b_{\alpha}-a^{3\alpha}_t (\p_t v_{\alpha}-b_0^{\nu}\p_{\nu}\p_t b_{\alpha})\\
&~=:g^*,\quad \text{on}\,\,\Gamma
\end{aligned}
\end{equation}
The standard elliptic estimate 
gives$$\|Q_t\|_{2.5}\lesssim\|f^*\|_{0.5}+\|g^*\|_{1,\Gamma}+\|Q_t\|_{0,\Gamma},$$
and by the multiplicative Sobolev inequality and trace lemma, one has
\begin{equation}\label{Qtinterior}
\begin{aligned}
\|f^*\|_{0.5}&\lesssim\epsilon\|Q_t\|_{2.5}+\|v\|_{2.5+\delta}(\|Q\|_{2.5}+\|v_t\|_{1.5})+\|v\|_2(\|v\|_{2.5+\delta}^2+\|v_t\|_{2.5}) \\
&~~~~+\|b_0\|_3\|b_t\|_{1.5}+\|b\|_{2.5+\delta}\|b_t\|_{1.5}+\|v\|_{2.5+\delta}\|b_0\|_{2.5}\|b\|_2
\lesssim\epsilon\|Q_t\|_{2.5}+ \PP ;
\end{aligned}
\end{equation}
\begin{equation}\label{Qtbdry}
\|g^*\|_{1,\Gamma}\lesssim\epsilon\|Q_t\|_{2.5}+\|v_{tt}\|_{1.5}+\|b_0\|_2\|b_t\|_{2.5}\lesssim \epsilon\|Q_t\|_{2.5}+\PP .
\end{equation}

For the boundary control, we first derive the expression of $Q_t|_{\Gamma}$. Time differentiating the equation $a^{33}Q=-\sigma\p_i(\sqrt{g} g^{ij}\p_j\eta^3$) we get:
\begin{equation}\label{Qtbd2}
Q_t=(1-a^{33})Q_t-a_t^{33}Q-\sigma\p_i(\p_t(\sqrt{g} g^{ij})\p_j \eta^3)-\sigma\p_1(\sqrt{g} g^{ij} \p_j v^3).
\end{equation}
By H\"older's inequality, Sobolev embedding and trace lemma, one can mimic the proof of Proposition 3.2 in \cite{disconzi2017prioriI} to get
\begin{equation}
\|Q_t\|_{0,\Gamma}\lesssim\|v\|_{2.5}.
\end{equation}

Therefore, summing up \eqref{Qtinterior}. \eqref{Qtbdry} and \eqref{Qtbd2}, then absorbing the $\epsilon$-term to LHS, one can get the bound for $Q_t$ as shown in \eqref{estimatesofQt}.

\subsection{$H^1$ control of $Q_{tt}$: Low regularity elliptic estimate}
In this section we will derive the  $H^1$ estimate of $Q_{tt}$. Although $Q_{tt}$ satisfies an elliptic PDE as $Q$ and $Q_t$, the standard elliptic, i.e. $\|u\|_{s}\lesssim \|f\|_{s-2}+\|g\|_{s-1.5,\p}+\|u\|_{0}$ is valid only for $s\geq 2$. Therefore we need to invoke the $H^1$ elliptic estimate in Lemma \ref{H1elliptic}. Since the RHS of the first equation in \eqref{lowelliptic} is required to be the divergence form, we need to start with the first equation in \eqref{MHDL} to derive the elliptic equation of $Q_{tt}$ instead of merely taking a time derivative in \eqref{Qtin}-\eqref{Qtbd}. 

Contracting the first equation of \eqref{MHDL} with $a^{\nu\alpha}\p_{\nu}$, invoking Piola's identity $\p_{\nu}a^{\nu\alpha}=0$, and then taking time derivative twice, we get 
\begin{equation}\label{Qttin}
\begin{aligned}
\p_{\nu}(a^{\nu\alpha}a^{\mu}_{\alpha}\p_{\mu}Q_{tt})&=\p_{\nu}\left(-\p_{tt}(a^{\nu\alpha}a^{\mu}_{\alpha})\p_{\mu}Q-2\p_t(a^{\nu\alpha}a^{\mu}_{\alpha})\p_{\mu}Q_t+\p_{tt}(a^{\nu\alpha}_t v_{\alpha})\right)\\
&~~~~+\p_{\nu}\left(a^{\nu\alpha}_{tt}b_0^{\mu}\p_{\mu}b_{\alpha}+2a_t^{\nu\alpha}b_0^{\mu}\p_{\mu}\p_t b_{\alpha}+a^{\nu\alpha}b_0^{\mu}\p_{\mu}\p_{tt}b_{\alpha}\right),
\end{aligned}
\end{equation}
with the boundary condition 
\begin{equation}\label{Qttbd}
\begin{aligned}
a^{\nu\alpha}a^{\mu}_{\alpha}\p_{\mu}Q_{tt} N_{\nu}&=\left(-\p_{tt}(a^{\nu\alpha}a^{\mu}_{\alpha})\p_{\mu}Q-2\p_t(a^{\nu\alpha}a^{\mu}_{\alpha})\p_{\mu}Q_t+\p_{tt}(a^{\nu\alpha} \p_t v_{\alpha})\right)N_{\nu}\\
&~~~~+\left(a^{\nu\alpha}_{tt}b_0^{\mu}\p_{\mu}b_{\alpha}+2a_t^{\nu\alpha}b_0^{\mu}\p_{\mu}\p_t b_{\alpha}+a^{\nu\alpha}b_0^{\mu}\p_{\mu}\p_{tt}b_{\alpha}\right)N_{\nu}.
\end{aligned}
\end{equation}
Let $A^{\mu\nu}=a^{\nu\alpha}a^{\mu}_{\alpha}$, $W=Q_{tt}$, and
\[
\begin{aligned}
\pi^{\nu}&=-\p_{tt}(a^{\nu\alpha}a^{\mu}_{\alpha})\p_{\mu}Q-2\p_t(a^{\nu\alpha}a^{\mu}_{\alpha})\p_{\mu}Q_t+\p_{tt}(a^{\nu\alpha}_t v_{\alpha})\\
&~~~~+a^{\nu\alpha}_{tt}b_0^{\mu}\p_{\mu}b_{\alpha}+2a_t^{\nu\alpha}b_0^{\mu}\p_{\mu}\p_t b_{\alpha}+a^{\nu\alpha}b_0^{\mu}\p_{\mu}\p_{tt}b_{\alpha},
\end{aligned}
\]and 
\[
\begin{aligned}
h~&=(-\p_{tt}(a^{\nu\alpha}a^{\mu}_{\alpha})\p_{\mu}Q-2\p_t(a^{\nu\alpha}a^{\mu}_{\alpha})\p_{\mu}Q_t+\p_{tt}(a^{\nu\alpha} \p_t v_{\alpha})\\
&~~~~+a^{\nu\alpha}_{tt}b_0^{\mu}\p_{\mu}b_{\alpha}+2a_t^{\nu\alpha}b_0^{\mu}\p_{\mu}\p_t b_{\alpha}+a^{\nu\alpha}b_0^{\mu}\p_{\mu}\p_{tt}b_{\alpha})N_{\nu}.
\end{aligned}
\] Then  \eqref{Qttin}-\eqref{Qttbd} exactly has the form as in \eqref{lowelliptic}. Before adapting Lemma \ref{H1elliptic} to the equation of $Q_{tt}$, we need to verify that $\pi$ and $\dive \pi$ are $L^2$-integrable. Repeatedly using H\"older's inequality and Sobolev embedding and Lemma \ref{estimatesofa}, we have
\begin{equation}\label{pi}
\begin{aligned}
\|\pi\|_{L^2}&\lesssim (\|a_{tt}\|_{L^3}\|a\|_{L^{\infty}}+\|a_t\|_{L^6}^2)\|\p Q\|_{L^6}+\|a\|_{L^6} \|a_t\|_{L^6} \|\p Q_t\|_{L^6}\\
&~~~~+\|a_{ttt}\|_{L^2}\|v\|_{L^{\infty}}+\|a_t\|_{L^3}\|v_t\|_{L^6}+\|a\|_{L^3}\|v_{tt}\|_{L^6} \\
&~~~~+\|a_{tt}\|_{L^2}\|b_0\|_{L^{\infty}}\|\p b\|_{L^{\infty}}+\|a_t\|_{L^6} \|b_0\|_{L^{\infty}}\|\p b_t\|_{L^3}+\|a\|_{L^{\infty}}\|b_0\|_{L^6}\|\p b_{tt}\|_{L^3} \\
&\lesssim (\|v_t\|_1+\|\p Q\|_1)(\|v\|_{1.5}\|v\|_{2.5+\delta}+\|v_t\|_{1.5})+\|v\|_{2.5+\delta}(\|Q_t\|_1+\|v_{tt}\|_{0})+\|v_{ttt}\|_0\\
&~~~~+\|v\|_{2.5+\delta}(\|v\|_{2.5+\delta}^2 \|v\|_1+\|v\|_{2.5+\delta}\|v_t\|_1+\|v_{tt}\|_1)\\
&~~~~+\|v_t\|_1\|b_0\|_2\|b\|_{2.5+\delta}+\|v\|_{1.5}\|b_t\|_{1.5}+\|b_0\|_1\|b_{tt}\|_{1.5}\\
&\lesssim \PP.
\end{aligned}
\end{equation}
Next we verify that $\dive\pi\in L^2$. From \eqref{Qin} and \eqref{Qttin}, we have 
\begin{equation}
\dive \pi=\p_{tt}(a_t^{\nu\alpha}\p_{\nu}v_{\alpha}+a^{\nu\alpha}\p_{\nu}b_0^{\mu}\p_{\mu}b_{\alpha}+\p_{\beta}b_{\gamma}a^{\nu\gamma}a^{\beta\alpha}\p_{\nu}b_{\alpha}-\p_{\beta}b_0^{\mu}a^{\beta\alpha}\p_{\mu}b_{\alpha}).
\end{equation} One can expand all terms and repeatly using H\"older's inequality, Sobolev embedding and Lemma \ref{estimatesofa} to get
\begin{equation}\label{divpi}
\begin{aligned}
\|\dive\pi\|_{L^2}&\lesssim\|a_{ttt}\|_{L^2}\|v\|_{2.5+\delta}+\|a_{tt}\|_{L^6}(\|\p v_t\|_{L^3}+\|\p b_0\|_{L^6}\|\p b\|_{L^6})\\
&~~~~+\|a_t\|_{L^6}(\|\p v_{tt}\|_{L^3}+\|\p b\|_{L^6}\|\p b_t\|_{L^6})+\|a\|_{L^{\infty}}(\|\p b\|_{L^{\infty}}\|\p b_{tt}\|_{L^2})\\
&\lesssim \PP.
\end{aligned}
\end{equation}
Now, Lemma \ref{H1elliptic} is valid for \eqref{Qttin}-\eqref{Qttbd} and yields that 
\begin{equation}\label{Qtt1}
\|Q_{tt}\|_1\lesssim \|\pi\|_0+\|h-\pi\cdot N\|_{-0.5,\Gamma}+\|Q_{tt}\|_{0,\Gamma}
\end{equation}where we use Lemma \ref{divtrace} for $Tr_N: L_{\dive}^2(\Omega)\to H^{-1/2}(\p\Omega)$ to get $$\|h-\pi\cdot N\|_{-0.5,\Gamma}=\|\p_t^3(a^{\nu\alpha}v_{\alpha})N_{\nu}\|_{-0.5,\Gamma}\lesssim \sum_{\nu}\|\p_t^3(a^{\nu\alpha}v_{\alpha})N_{\nu}\|_0.$$ This is valid because $\p_{\nu}(\p_t^3(a^{\nu\alpha}v_{\alpha}))=0\in L^2$.

It remains to control $\|Q_{tt}\|_{0,\Gamma}$. One can differentiate $\p_t$ twice to the surface tension equation on the boundary, i.e. the fifth equation in \eqref{MHDL}, to get
\begin{equation}
\begin{aligned}
Q_{tt}&=(1-a^{33})Q_{tt}-\p_t^2a^{33}Q-2a^{33}_tQ_t \\
&~~~~-\sigma\p_i(\sqrt{g}g^{ij}\p_j\eta^3_{tt})-\sigma\p_i(\p_t^2(\sqrt{g}g^{ij})\p_j\eta^3)-2\sigma\p_i(\p_t(\sqrt{g}g^{ij})\p_j\p_t\eta^3).
\end{aligned}
\end{equation}
Therefore, it suffices to control the $L^2(\Gamma)$ norm of each term on RHS. The terms containing $q$ are all easy to control by H\"older's inequality, Sobolev embedding and trace lemma:
\[
\|(1-a^{33})Q_{tt}\|_{0,\Gamma}+\|\p_t^2a^{33}Q\|_{0,\Gamma}+\|2a^{33}_tQ_t \|_{0,\Gamma}\lesssim\epsilon\|Q_{tt}\|_1+\|v_t\|_{1.5}\|Q\|_{2.5}+\|v\|_2\|Q_t\|_1.
\] 
For the $L^2(\Gamma)$-estimate of $-\sigma\p_i(\sqrt{g}g^{ij}\p_j\eta^3_{tt})$, we have:
\[
\|\p_i(\sqrt{g}g^{ij}\p_j\eta^3_{tt})\|_{0,\Gamma}\lesssim\|v_t\|_{2.5},
\]
where we refer to Proposition 3.2 in \cite{disconzi2017prioriI} for detailed computation.

However, the $L^2(\Gamma)$-estimates of $\p_i(\p_t^2(\sqrt{g}g^{ij})\p_j\eta^3)$ and $\p_i(\p_t(\sqrt{g}g^{ij})\p_j\p_t\eta^3)$ need to be refined in order to make us easier to write the pressure estimates in terms of the sum of initial data and time integral of $\PP$ when we close all the a priori estimates. First, we have
\[
\begin{aligned}
\|\p_i(\p_t^2(\sqrt{g}g^{ij})\p_j\eta^3)\|_{0,\Gamma}&\leq \|\p_i(\p_t^2(\sqrt{g}g^{ij}))\p_j\eta^3\|_{0,\Gamma}+\|\p_t^2(\sqrt{g}g^{ij})\p_i\p_j\eta^3\|_{0,\Gamma}\\
&\lesssim\|\p_t^2\TP(\sqrt{g}g^{-1})\|_{0,\Gamma}+\|\p_t^2(\sqrt{g}g^{ij})\|_{0,\Gamma}\|\TP^2\eta\|_{L^{\infty}(\Gamma)}.
\end{aligned}
\]
Then we write the derivatives of $\sqrt{g}g^{ij}$ in terms of $R(\TP\eta)$, a rational function of $\p_i \eta$ sstisfying $\|R(\TP\eta)\|_{1.5,\Gamma}\lesssim\|\TP\eta\|_{1.5,\Gamma}$ (For the detailed illustration, see Remark 2.4 in \cite{disconzi2017prioriI}):
\[
\p_t^2\TP(\sqrt{g}g^{-1})=R(\TP\eta)(\TP v)^2\TP^2\eta+R(\TP\eta)\TP v_t +R(\TP\eta)\TP v\TP^2 v+R(\TP\eta)\TP^2v_t,
\]and 
\[
\p_t^2(\sqrt{g}g^{-1})=R(\TP\eta)\TP v_t+R(\TP\eta)(\TP v)^2.
\]Invoking Lemma \ref{dhj}, we have
\[
\begin{aligned}
\|R(\TP\eta)(\TP v)^2\TP^2\eta+R(\TP\eta)\TP v_t\|_{0,\Gamma}&\lesssim\|R(\TP\eta)\|_{L^{\infty(\Gamma)}}\|\TP v\|_{L^4(\Gamma)}^2\|\TP^2\eta\|_{L^4(\Gamma)}\\
&\lesssim\|v\|_2 \\
\|R(\TP\eta)\TP v_t +Q(\TP\eta)\TP v\TP^2 v\|_{0,\Gamma}&\lesssim\|\TP v_t\|_{L^2(\Gamma)}\|\TP^2\eta\|_{L^{\infty(\Gamma)}} \\
&\lesssim\|v_t\|_{1.5}\|Q\|_2 \\
\|R(\TP\eta)\TP v\TP^2 v\|_{0,\Gamma}&\lesssim \|v\|_{2.5}\|v\|_{2.5+\delta} \\
\|R(\TP\eta)\TP^2v_t\|_{0,\Gamma}&\lesssim \|v_t\|_{2.5},
\end{aligned}
\] so
\begin{equation}
\|\p_i(\p_t^2(\sqrt{g}g^{ij}))\p_j\eta^3\|_{0,\Gamma}\lesssim\|v\|_2+\|v_t\|_{1.5}\|Q\|_2+\|v\|_{2.5}\|v\|_{2.5+\delta}+\|v_t\|_{2.5}.
\end{equation}
Similarly, we can get 
\begin{equation}
\|\p_t^2(\sqrt{g}g^{ij})\p_i\p_j\eta^3\|_{0,\Gamma}\lesssim\|v_t\|_{1.5}\|Q\|_2+\|v_t\|_2+\|v\|_2^2\|Q\|_{2}.
\end{equation}
Moreover, since $\p_t (\sqrt{g}g^{-1})=R(\cp \eta)(\cp v)^2$, we have
\begin{equation}
\|\p_i(\p_t(\sqrt{g}g^{ij})\p_j\p_t\eta^3)\| \lesssim \|v\|_{2.5}.
\end{equation}

Summing up all the boundary terms of $Q_{tt}$, one gets
\begin{equation}\label{Qttbd2}
\|Q_{tt}\|_{0,\Gamma}\lesssim\|v_t\|_{2.5}+\|v\|_2\|Q_t\|_1+(\|v_t\|_{1.5}+\|v\|_2^2)\|Q\|_{2.5}+\|v\|_{2.5+\delta}\|v\|_{2.5}.
\end{equation}Therefore, combining \eqref{Qtt1} and \eqref{Qttbd2}, and absorbing the $\epsilon$-terms to LHS, we get the $H^1$ estimate of $Q_{tt}$ as shown in \eqref{estimatesofQtt}.

\begin{flushright}
$\square$
\end{flushright}
\section{Div-Curl Estimates}\label{divcurl}

In this section we derive the div-curl estimates of $v$ and $b$ and those of their time derivatives as the first step to derive the desired a priori estimates. Specifically, we show:
\begin{prop}\label{dc}
Assume the assumptions of Lemma \ref{estimatesofa} holds, we have the following estimates:
\begin{equation}\label{divcurlvb00}
\begin{aligned}
\|v\|_{3.5}&\lesssim\PP_0+\int_0^t \PP +\|v^3\|_{3,\Gamma},\\
\|b\|_{3.5}&\lesssim\PP_0+\int_0^t \PP;
\end{aligned}
\end{equation}
and
\begin{equation}\label{divcurlvtbt00}
\begin{aligned}
\|v_t\|_{2.5}&\lesssim\PP_0+\int_0^t \PP +\|v_t^3\|_{2,\Gamma},\\
\|b_t\|_{2.5}&\lesssim\PP_0+\int_0^t \PP;
\end{aligned}
\end{equation}
and
\begin{equation}
\begin{aligned}\label{divcurlvttbtt00}
\|v_{tt}\|_{1.5}&\lesssim\PP_0+\int_0^t \PP~ds +\|v_{tt}^3\|_{1,\Gamma}+P(\| v\|_{2.5+\delta}),\\
\|b_{tt}\|_{1.5}&\lesssim\PP_0+\int_0^t \PP~ds+P(\| v\|_{2.5+\delta},\|b\|_{2.5+\delta}),
\end{aligned}
\end{equation}where $\delta>0$ is a constant to be determined, and can be arbitratily small.
\end{prop}

The basic tool is Hodge's decomposition inequality, i.e. for any (smooth) vector field $X$, it holds 
$$\|X\|_s\lesssim \|X\|_0+\|\curl X\|_{s-1}+\|\dive X\|_{s-1}+\|(X\cdot N)\|_{s-1/2,\Gamma};$$
where $N$ is the outer unit normal vector to $\Gamma$.
This inequality will be used to control $v,v_t,v_{tt}$. 
In addition, we mention here that since $v^3=0$ on $\Gamma_0$, $v^3,v_t^3$ ,$v_{tt}^3$ also vanish on $\Gamma_0$.

\subsection{Div-Curl estimates of $v$ and $b$}\label{dcvb}
We adopt the following notations throughout the rest of this section. 
\nota 
Let $X$ be any smooth vector field. We define
\begin{align*}
A_a X=a^{\mu\alpha}\p_\mu X_\alpha,\q  A_I X=\di X= \delta^{\mu\alpha}\p_\mu X_\alpha,\\
(B_a X)^\gamma = \epsilon^{\gamma \alpha \beta} a^{\mu}_{\alpha}\p_\mu X_\beta,\q (B_I X)^\gamma= \curl X= \epsilon^{\gamma\alpha\beta} \p_\alpha X_\beta.
\end{align*}
Here, $\epsilon^{\gamma\alpha\beta}$ is the totally anti-symmetric symbol with $\epsilon^{123}=1$. 
In other words, we use $A_a$ and $B_a$ to denote the Eulerian divergence and curl operators, respectively. \\

From Hodge's decomposition inequality applied to $v$ and $b$, we have:
\begin{equation}
\begin{aligned}\label{divcurlvb0}
\|v\|_{3.5}&\lesssim\|v\|_0+\|\curl v\|_{2.5}+\|\dive v\|_{2.5}+\|v^3\|_{3,\Gamma}, \\
\|b\|_{3.5}&\lesssim\|b\|_0+\|\curl b\|_{2.5}+\|\dive b\|_{2.5}+\|b^3\|_{3,\Gamma}.
\end{aligned}
\end{equation}

First, the divergence control is easy.  From Lemma \ref{estimatesofa} (7), we know it holds in a sufficiently short time $[0,T]$ that
\begin{equation}
\begin{aligned}\label{divvb}
\|\dive v\|_{2.5}&=\|\underbrace{A_av}_{=0}+(A_I-A_a)v\|_{2.5}\lesssim \|I-a\|_{2.5}\|v\|_{3.5}\lesssim\epsilon\|v\|_{3.5}\\
\|\dive b\|_{2.5}&=\|\underbrace{A_ab}_{=0}+(A_I-A_a)b\|_{2.5}\lesssim \|I-a\|_{2.5}\|b\|_{3.5}\lesssim\epsilon\|b\|_{3.5}.
\end{aligned}
\end{equation}

The control of $\curl v$ and $\curl b$ follows exactly in the same way as Proposition 5.2 in \cite{luozhang2019MHD2.5}, just replacing $\p^{1.5}$ in that paper by $\p^{2.5}$. We have:
\begin{equation}\label{curlvb}
\|\curl v\|_{2.5}+\|\curl b\|_{2.5}\lesssim\epsilon(\|v\|_{3.5}+\|b\|_{3.5})+\PP_0+\int_0^t \PP.
\end{equation}

Now we are going to control $\|b^3\|_{3,\Gamma}$. Recall that $b=(b_0\cdot\p)\eta$ and $b_0^3=0$ on the boundary $\Gamma$, we know $b=(b_0\cdot\TP)\eta$ on the boundary $\Gamma$. Invoking Lemma \ref{dhj} and trace lemma, we are able to get
\begin{equation}\label{bbdry}
\|b^3\|_{3,\Gamma}= \|b_0\cdot\TP\eta\|_{3,\Gamma}\lesssim\|b_0\|_{3,\Gamma}\|\eta\|_{4,\Gamma}\lesssim\|b_0\|_{3,\Gamma}\|Q\|_{2,\Gamma}\lesssim P(\|b_0\|_{3.5},\|Q(0)\|_{2.5})+\int_0^t\|Q_t\|_{2.5}.
\end{equation}

Combining \eqref{divcurlvb0}, \eqref{divvb}, \eqref{curlvb}, \eqref{bbdry}, and absorbing the $\epsilon$-term to LHS, we conclude that 
\begin{equation}\label{divcurlvb}
\begin{aligned}
\|v\|_{3.5}&\lesssim\PP_0+\int_0^t \PP +\|v^3\|_{3,\Gamma};\\
\|b\|_{3.5}&\lesssim\PP_0+\int_0^t \PP.
\end{aligned}
\end{equation}

\subsection{Div-Curl estimates of $v_t$ and $b_t$}\label{dcvtbt}
Again, from Hodge's decomposition inequality applied to $v_t$ and $b_t$, we have:
\begin{equation}
\begin{aligned}\label{divcurlvtbt0}
\|v_t\|_{2.5}&\lesssim\|v_t\|_0+\|\curl v_t\|_{1.5}+\|\dive v_t\|_{1.5}+\|v_t^3\|_{2,\Gamma}; \\
\|b_t\|_{2.5}&\lesssim\|b_t\|_0+\|\curl b_t\|_{1.5}+\|\dive b_t\|_{1.5}+\|b_t^3\|_{2,\Gamma},
\end{aligned}
\end{equation}where $ \mathcal{T}$ is any unit tangential vector to $\Gamma$.

To control the divergence, we again invoke $A_a v=A_a b=0$ to get:
$$\dive v_t=A_a v_t+(A_I-A_a)v_t=\p_t(\underbrace{A_a v}_{=0})-A_{a_t}v+(A_I-A_a)v_t=-A_{a_t}v+(A_I-A_a)v_t;$$
$$\dive b_t=A_a b_t+(A_I-A_a)b_t=\p_t(\underbrace{A_a b}_{=0})-A_{a_t}b+(A_I-A_a)b_t=-A_{a_t}b+(A_I-A_a)b_t.$$ 
Therefore, one can use the multiplicative Sobolev inequality and Lemma \ref{estimatesofa} to get
\begin{equation}
\begin{aligned}\label{divvt}
\|\dive v_t\|_{1.5}&=\|A_{a_t}v\|_{1.5}+\|(A_I-A_a)v_t\|_{1.5}\\
&\lesssim\|{a_t}\|_{1.5}\|v\|_{2}+\|I-a\|_{1.5}\|v_t\|_{2.5}\\
&\lesssim\|\eta\|_{2.5}^4\|v\|_{2.5}\|v\|_{2}+\epsilon \|v_t\|_{2.5}\\
&\lesssim P(\|v_0\|_{2.5})+\int_0^t P(\|v_t(s)\|_{2.5})ds+\epsilon \|v_t\|_{2.5},
\end{aligned}
\end{equation} and similarly,
\begin{equation}\label{divbt}
\|\dive b_t\|_{1.5}\lesssim P(\|b_0\|_{2.5})+\int_0^t P(\|b_t(s)\|_{2.5})ds+\epsilon \|b_t\|_{2.5}.
\end{equation}

Now we start to control $\curl v_t$ and $\curl b_t$.  First, we have
\begin{equation}\label{curlvtbt0}
\begin{aligned}
\|\curl v_t\|_{1.5}&\leq \|B_a v_t\|_{1.5}+\|(B_I- B_a)v_t\|_{1.5}\lesssim \|B_a v_t\|_{1.5}+\epsilon \|v_t\|_{2.5}\\
\|\curl b_t\|_{1.5}&\leq \|B_a b_t\|_{1.5}+\|(B_I- B_a)b_t\|_{1.5}\lesssim\|B_a b_t\|_{1.5}+\epsilon \|b_t\|_{2.5}.
\end{aligned}
\end{equation}
The control of $B_a v_t$ and $B_a b_t$ is slightly different from that of $B_a v$ and $B_a b$. We start with the first equation of \eqref{MHDL} $$v_t^{\alpha}=(b_0\cdot\p)^2\eta^{\alpha}-a^{\nu\alpha}\p_{\nu}Q.$$ Taking the time derivative at first, and then apply $B_a$ on both sides, we get
$$\p_t(B_a v_t)_{\lambda}-(B_a(b_0\cdot \p)^2 v)_{\lambda}=\underbrace{(B_{a_t} v)_{\lambda}-\epsilon_{\lambda\tau\alpha}a^{\mu\tau}\p_{\mu}(a^{\nu\alpha}_t\p_{\nu}Q)}_{G^*}.$$
Commuting $(b_0\cdot\p)$ with $B_a$ on LHS, we have 
$$\p_t(B_a v_t)_{\lambda}-(b_0\cdot \p)(B_a(b_0\cdot \p) v)_{\lambda}=G^*+[B_a,b_0\cdot\p]b_t.$$
Taking $\p^{1.5}$ on both sides and commuting $b_0\cdot\p$ with $B_a$, we get the evolution equation of $\curl v_t$:
\begin{equation}\label{curlvtbteq}
\p_t\p^{1.5}(B_av_t)-(b_0\cdot\p)\p^{1.5}(B_a b_t)=\underbrace{\p^{1.5}(G^*+[B_a,b_0\cdot\p]b_t)+[\p^{1.5},b_0\cdot\p]B_a b_t}_{F^*}. 
\end{equation} Here, we use the second equation of \eqref{MHDL} and \eqref{GW2}, i.e., $b_t=(b_0\cdot\p) v$. Next we again mimic the proof of Proposition 5.2 in \cite{luozhang2019MHD2.5} and get 
\begin{equation}
\begin{aligned}\label{curlvtbt1}
\frac{1}{2}\frac{d}{dt}\int_{\Omega}|\p^{1.5} B_a v_t|^2+|\p^{1.5} B_a b_t|^2dy
=\underbrace{\int_{\Omega}F^*\cdot \p^{1.5}B_a v_tdy}_{\mathcal{B}_1^*}\\
+\underbrace{\int_{\Omega}\p^{1.5} (B_a b_t)\cdot [\p^{1.5}B_a,b_0\cdot\p]v_tdy }_{\mathcal{B}_2^*}\\
~~~~+\underbrace{\int_{\Omega}\p^{1.5}(B_a b_t)^{\lambda}\p^{1.5}(\epsilon_{\lambda\tau\alpha}a^{\mu\tau}_t\p_{\mu}b_t^{\alpha})dy}_{\mathcal{B}_3^*}.
\end{aligned}
\end{equation}
$\mathcal{B}_3^*$ can be controlled directly by the multiplicative Sobolev inequality:
\begin{equation}\label{B3*}
\begin{aligned}
\mathcal{B}_3^*&\lesssim \|\p^{1.5}B_a b_t\|_{0}\|\p^{1.5}(\epsilon_{\lambda\tau\alpha}a^{\mu\tau}_t\p_{\mu}b_t^{\alpha})\|_0\\
&\lesssim \|a\|_2\|b_t\|_{2.5}\|a_t\|_2\|b_t\|_{2.5}\lesssim \|v\|_{3}\|b_t\|_{2.5}^2.
\end{aligned}
\end{equation}
To control $\mathcal{B}_2^*$, it suffices to control $\|[\p^{1.5}B_a,b_0\cdot\p]v_t\|_{L^2}$. First we simplify the commutator:
\begin{equation}\label{B2*0}
\begin{aligned}
[\p^{1.5}B_a,b_0\cdot\p]v_t&=\epsilon_{\lambda\tau\alpha}\left(\p^{1.5}(a^{\mu\tau}\p_{\mu}(b_0^{\nu}\p_{\nu}v^{\alpha}_t))-b_0^{\nu}\p_{\nu}\p^{1.5}(a^{\mu\tau}\p_{\mu}v^{\alpha}_t)\right) \\
&=\epsilon_{\lambda\tau\alpha}\underbrace{\left(\p^{1.5}(a^{\mu\tau}\p_{\mu}(b_0^{\nu}\p_{\nu}v^{\alpha}_t))-\p_{\nu}\p^{1.5}(b_0^{\nu}a^{\mu\tau}\p_{\mu}v^{\alpha}_t)\right)}_{\mathcal{B}_{21}^*}  \\
&~~~~+\epsilon_{\lambda\tau\alpha}\underbrace{\left(\p_{\nu}\p^{1.5}(b_0^{\nu}a^{\mu\tau}\p_{\mu}v^{\alpha}_t)-b_0^{\nu}\p_{\nu}\p^{1.5}(a^{\mu\tau}\p_{\mu}v^{\alpha}_t)\right)}_{\mathcal{B}_{22}^*}.
\end{aligned}
\end{equation}
For $\mathcal{B}_{22}^*$, we need to invoke the refind Kato-Ponce type commutator estimate \eqref{lidong2} because $H^{1.5}(\Omega)\nsubseteq L^{\infty}(\Omega)$.
\begin{equation}\label{B22*}
\|\mathcal{B}_{22}^*\|_{L^2}\lesssim \|b_0\|_{W^{1.5,3}}\|a^{\mu\tau}\p_{\mu}v_t^{\alpha}\|_{L^6}+\|\p b_0\|_{L^{\infty}}\|a^{\mu\tau}\p_{\mu}v_t^{\alpha}\|_{1.5}\lesssim \|b_0\|_3\|v_t\|_{2.5}.
\end{equation}
For $\mathcal{B}_{21}^*$, we have 
\begin{equation}\label{B21*0}
\begin{aligned}
\mathcal{B}_{21}^*&=\epsilon_{\lambda\tau\alpha}\p^{1.5}(a^{\mu\tau}\p_{\mu}(b_0^{\nu}\p_{\nu}v^{\alpha}_t))-\p_{\nu}(b_0^{\nu}a^{\mu\tau}\p_{\mu}v^{\alpha}_t))\\
&=\epsilon_{\lambda\tau\alpha}\p^{1.5}\left(a^{\mu\tau}\p_{\mu}b_0^{\nu}\p_{\nu}v^{\alpha}_t+a^{\mu\tau}b_0^{\nu}\p_{\mu}\p_{\nu}v^{\alpha}_t-b_0^{\nu}\p_{\nu}a^{\mu\tau}\p_{\mu}v^{\alpha}_t-b_0^{\nu}a^{\mu\tau}\p_{\mu}\p_{\nu}v^{\alpha}_t\right) \\
&=\epsilon_{\lambda\tau\alpha}\p^{1.5}\left(a^{\mu\tau}\p_{\mu}b_0^{\nu}\p_{\nu}v^{\alpha}_t+b_0^{\nu}\p_{\beta}\p_{\nu}\eta_{\gamma}a^{\mu\gamma}a^{\beta\tau}\p_{\mu}v^{\alpha}_t \right)\\
&=\epsilon_{\lambda\tau\alpha}\p^{1.5}(a^{\mu\tau}\p_{\mu}b_0^{\nu}\p_{\nu}v^{\alpha}_t+\p_{\beta}((b_0\cdot\p)\eta_{\gamma})a^{\mu\gamma}a^{\beta\tau}\p_{\mu}v^{\alpha}_t-\p_\beta b_0^{\mu} a^{\beta\tau}\p_\mu v^\alpha_t ),
\end{aligned}
\end{equation}where we used Lemma \ref{estimatesofa} (8) to expand $b_0^\nu \p_\nu a^{\mu\tau}\p_\mu v^\alpha$ in the second line and $\p_{\nu}\eta_{\gamma}a^{\mu\gamma}=\delta^{\mu}_{\nu}$. 
Therefore, invoking $b=(b_0\cdot\p)\eta$ and the multiplicative Sobolev inequality again, one can get:
\begin{equation}\label{B21*}
\|\mathcal{B}_{21}^*\|_{L^2}\lesssim\|b_0\|_3\|v_t\|_{2.5}.
\end{equation}
It remains to control $\mathcal{B}_1^*$, specifically, $\|F^*\|_{L^2}$. The two commutator terms can be controlled in the same way as $\mathcal{B}_{21}^*$ and straightforward computation (we omit the computation details):
\begin{equation}\label{F*1}
\|[\p^{1.5},b_0\cdot\p]B_a b_t\|_0+\|\p^{1.5}([B_a,b_0\cdot\p]b_t)\|_0\lesssim\|b_0\|_3\|v_t\|_{2.5}.
\end{equation} For $G^*=(B_{a_t} v)_{\lambda}-\epsilon_{\lambda\tau\alpha}a^{\mu\tau}\p_t(a^{\nu\alpha}\p_{\nu}Q)$, the multiplicative Sobolev inequality combined with Lemma \eqref{estimatesofa} yields that
\begin{equation}\label{G*}
\|B_{a_t}v\|_{1.5}+\|a^{\mu\tau}\p_{\mu}(a^{\nu\alpha}_t\p_{\nu}Q)\|_{1.5}\lesssim\|v\|_3(\|v\|_{3.5}+\|Q\|_{3.5})+\|v\|_{3.5}\|Q\|_{3}.
\end{equation}
Combining \eqref{curlvtbt0}, \eqref{curlvtbt1}, \eqref{B3*}, \eqref{B22*}, \eqref{B21*}, \eqref{F*1} and \eqref{G*}, and absorbing the $\epsilon$-term to LHS we have:
\begin{equation}\label{curlvtbt}
\|\curl v_t\|_{1.5}+\|\curl b_t\|_{1.5}\lesssim \PP_0+\int_0^t\PP.
\end{equation}

The boundary term $\|b_t^3\|_{2,\Gamma}$ can be similarly controlled as $\|b^3\|_{3,\Gamma}$:
\begin{equation}\label{btbdry}
\|b_t^3\|_{2,\Gamma}=\|b_0\cdot\TP \p_t\eta\|_{2,\Gamma}\lesssim\|b_0\|_{2,\Gamma}\|v\|_{3,\Gamma}\lesssim P(\|b_0\|_{2.5},\|v_0\|_{3.5})+\int_0^t \|v_t\|_{2.5}
\end{equation}

Summing up \eqref{divvt}, \eqref{divbt}, \eqref{curlvtbt} and \eqref{btbdry}, then absorbing the $\epsilon$-term to LHS, we have
\begin{equation}\label{divcurlvtbt}
\begin{aligned}
\|v_t\|_{2.5}&\lesssim\PP_0+\int_0^t \PP +\|v_t^3\|_{2,\Gamma};\\
\|b_t\|_{2.5}&\lesssim\PP_0+\int_0^t \PP.
\end{aligned}
\end{equation}

\subsection{Div-Curl estimates of $v_{tt}$ and $b_{tt}$}

Again, from Hodge's decomposition inequality applied to $v_{tt}$ and $b_{tt}$, we have:
\begin{equation}
\begin{aligned}\label{divcurlvttbtt0}
\|v_{tt}\|_{1.5}&\lesssim\|v_{tt}\|_0+\|\curl v_{tt}\|_{0.5}+\|\dive v_{tt}\|_{0.5}+\|v_{tt}^3\|_{1,\Gamma}; \\
\|b_{tt}\|_{1.5}&\lesssim\|b_{tt}\|_0+\|\curl b_{tt}\|_{0.5}+\|\dive b_{tt}\|_{0.5}+\|b_{tt}\cdot\mathcal{T}\|_{1,\Gamma},
\end{aligned}
\end{equation}where $ \mathcal{T}$ is any unit tangential vector to $\Gamma$.
To control the divergence, we again invoke $A_a v=A_a b=0$ to get:
$$\dive v_{tt}=A_a v_{tt}+(A_I-A_a)v_{tt}=\p_{tt}(\underbrace{A_a v}_{=0})-A_{a_{tt}}v-2A_{a_t}v_t+(A_I-A_a)v_{tt};$$
$$\dive b_{tt}=A_a b_{tt}+(A_I-A_a)b_{tt}=\p_{tt}(\underbrace{A_a b}_{=0})-A_{a_{tt}}b-2A_{a_t}b_t+(A_I-A_a)b_{tt}.$$ 
Therefore, one can use the multiplicative Sobolev inequality and Lemma \ref{estimatesofa} to get
\begin{equation}
\begin{aligned}\label{divvtt}
\|\dive v_{tt}\|_{0.5}&\leq\|A_{a_{tt}}v\|_{0.5}+2\|A_{a_t}v_t\|_{0.5}+\|I-a\|_2\|v_{tt}\|_{1.5}\\
&\lesssim \|a_{tt}\|_{0.5}\|v\|_{2.5+\delta}+\|a_t\|_{L^{\infty}}\|v_t\|_{1.5}+\epsilon\|v_{tt}\|_{1.5}\\
&\lesssim \|\p v\|_{0.5}\|v\|_{2.5+\delta}^2+\|\p v_t\|_{0.5}\|v\|_{2.5+\delta}+\|v\|_{2.5+\delta}\|v_t\|_{1.5}+\epsilon\|v_{tt}\|_{1.5}\\
&\lesssim P(\| v\|_{2.5+\delta})(\|v\|_{1.5}+\|v_t\|_{1.5})+\epsilon\|v_{tt}\|_{1.5}
\end{aligned}
\end{equation} and similarly,
\begin{equation}\label{divbtt}
\|\dive b_{tt}\|_{0.5}\lesssim P(\| v\|_{2.5+\delta},\|b\|_{2.5+\delta})(\|v\|_{1.5}+\|v_t\|_{1.5}+\|b_t\|_{1.5})+\epsilon\|b_{tt}\|_{1.5},
\end{equation}where $\delta>0$ can be arbitratily small.

The boundary term $\|b_{tt}^3\|_{1,\Gamma}$ is again controlled in the same way as \eqref{btbdry}
\begin{equation}\label{bttbdry}
\|b_{tt}^3\|_{1,\Gamma}=\|b_0\cdot\TP v_{t}\|_{1,\Gamma}\lesssim P(\|b_0\|_{2.5},\|v_{t}(0)\|_{1.5})+\int_0^t \|v_{tt}\|_{1.5}.
\end{equation}

Apart from $\|v_{tt}^3\|_{1,\Gamma}$, it remains to control $\curl v_{tt}$ and $\curl b_{tt}$.
We have:
\begin{equation}
\begin{aligned}\label{curlvttbtt0}
\|\curl v_{tt}\|_{0.5}&=\|B_a v_{tt}+(B_I-B_a)v_{tt}\|_{0.5}\leq\|B_a v_{tt}\|_{0.5}+\epsilon\|v_{tt}\|_{1.5}\\
\|\curl b_{tt}\|_{0.5}&=\|B_a b_{tt}+(B_I-B_a)b_{tt}\|_{0.5}\leq\|B_a b_{tt}\|_{0.5}+\epsilon\|b_{tt}\|_{1.5}.
\end{aligned}
\end{equation}
Applying the $B_a$ operator on both sides of the first equation in \eqref{MHDL}, we have $$B_a v_t=B_a (b_0\cdot\p)^2\eta.$$ Then taking time derivative twice, we get 
\begin{equation}\label{curlvttbtt00}
\p_t(B_a v_{tt})- B_a(b_0\cdot\p)^2 v_t=G^{**},
\end{equation} where we used \eqref{GW2} to derive 
\begin{equation}\label{G**}
G^{**}:=-B_{a_{tt}}v_t-B_{a_t}v_{tt}+B_{a_{tt}}(b_0\cdot\p)b+2 B_{a_t}(b_0\cdot\p)b_t.
\end{equation}
Commuting $(b_\cdot\p)$ with $B_a$ on LHS of \eqref{curlvttbtt00}, taking $\p^{0.5}$ derivative and then commuting it with $b_0\cdot\p$, we get the evolution equation of $B_a v_{tt}$ and $B_ab_{tt}$ with the help of \eqref{GW2}:
\begin{equation}\label{curlvttbtt000}
\p_t(\p^{0.5}B_a v_{tt})- (b_0\cdot\p)(\p^{0.5}B_a b_{tt})=\underbrace{\p^{0.5}(G^{**}+[B_a,b_0\cdot\p]b_{tt})+[\p^{0.5},b_0\cdot\p](B_a b_{tt})}_{=:F^{**}}.
\end{equation}
Analogous to \eqref{curlvtbt1}, we can derive the following energy identity:
\begin{equation}
\begin{aligned}
\frac{1}{2}\frac{d}{dt}\int_{\Omega}|\p^{0.5} B_a v_{tt}|^2+|\p^{0.5} B_a b_{tt}|^2dy
=\underbrace{\int_{\Omega}F^{**}\cdot \p^{0.5}B_a v_{tt}dy}_{\mathcal{B}_1^{**}}\\
+\underbrace{\int_{\Omega}\p^{0.5} (B_a b_{tt})\cdot [\p^{0.5}B_a,b_0\cdot\p]v_tdy }_{\mathcal{B}_2^{**}}+\underbrace{\int_{\Omega}\p^{0.5}(B_a b_{tt})^{\lambda}\p^{0.5}(\epsilon_{\lambda\tau\alpha}a^{\mu\tau}_t\p_{\mu}b_{tt}^{\alpha})dy}_{\mathcal{B}_3^{**}}.
\end{aligned}
\end{equation}
The multiplicative Sobolev inequality together with Lemma \ref{estimatesofa} yields that 
\begin{equation}\label{B3**}
\mathcal{B}_3^{**}\lesssim \|b_{tt}\|_{1.5}\|b_0\|_{1.5}\|v_t\|_{2.5}\|v\|_{2}.
\end{equation}
To control  $\mathcal{B}_2^{**}$, it suffices to control $\|[\p^{0.5}B_a,b_0\cdot\p]v_{tt}\|_{L^2}$. Analogous to \eqref{B2*0}, we have
\begin{equation}\label{B2**0}
\begin{aligned}
[\p^{0.5}B_a,b_0\cdot\p]v_{tt}&=\epsilon_{\lambda\tau\alpha}\underbrace{\left(\p^{0.5}(a^{\mu\tau}\p_{\mu}(b_0^{\nu}\p_{\nu}v^{\alpha}_{tt}))-\p_{\nu}\p^{0.5}(b_0^{\nu}a^{\mu\tau}\p_{\mu}v^{\alpha}_{tt})\right)}_{\mathcal{B}_{21}^{**}}  \\
&~~~~+\epsilon_{\lambda\tau\alpha}\underbrace{\left(\p_{\nu}\p^{0.5}(b_0^{\nu}a^{\mu\tau}\p_{\mu}v^{\alpha}_{tt})-b_0^{\nu}\p_{\nu}\p^{0.5}(a^{\mu\tau}\p_{\mu}v^{\alpha}_{tt})\right)}_{\mathcal{B}_{22}^{**}}.
\end{aligned}
\end{equation}
For $\mathcal{B}_{22}^{**}$, we need to invoke the refind Kato-Ponce type commutator estimate as in \eqref{B22*}
\begin{equation}\label{B22**}
\|\mathcal{B}_{22}^*\|_{L^2}\lesssim \|b_0\|_{W^{1.5,6}}\|a^{\mu\tau}\p_{\mu}v_{tt}^{\alpha}\|_{L^3}+\|\p b_0\|_{L^{\infty}}\|a^{\mu\tau}\p_{\mu}v_{tt}^{\alpha}\|_{1.5}\lesssim \|b_0\|_3\|v_{tt}\|_{1.5}.
\end{equation}
For $\mathcal{B}_{21}^{**}$, we have 
\begin{equation}\label{B21**0}
\mathcal{B}_{21}^{**}=\epsilon_{\lambda\tau\alpha}\p^{0.5}(a^{\mu\tau}\p_{\mu}b_0^{\nu}\p_{\nu}v^{\alpha}_{tt}+\p_{\beta}((b_0\cdot\p)\eta_{\gamma})a^{\mu\gamma}a^{\beta\tau}\p_{\mu}v^{\alpha}_{tt}-\p_\beta b_0^{\mu} a^{\beta\tau}\p_\mu v^\alpha_{tt} ),
\end{equation}
Therefore, invoking $b=(b_0\cdot\p)\eta$ and the multiplicative Sobolev inequality again, one can get:
\begin{equation}\label{B21**}
\|\mathcal{B}_{21}^*\|_{L^2}\lesssim\|b_0\|_3\|v_{tt}\|_{1.5}.
\end{equation}
It remains to control $\mathcal{B}_1^{**}$, specifically, $\|F^{**}\|_{L^2}$. The two commutator terms can be controlled by $\|b_0\|_3\|b_{tt}\|_{1.5}$ in the same way as $\mathcal{B}_{1}^*$. Therefore it remains to control $\|G^{**}\|_{0.5}$, which is directly controlled by using multiplicative Sobolev inequality
\begin{equation}\label{G**}
\|G^{**}\|_{0.5}\lesssim\|a_{tt}\|_1(\|v_t\|_2+\|b_0\|_3\|b\|_3)+\|a_t\|_2(\|v_{tt}\|_{1.5}\|b_0\|_3\|b_t\|_{2.5})\lesssim\PP.
\end{equation}
Combining \eqref{curlvtbt0}, \eqref{curlvtbt1}, \eqref{B3**}, \eqref{B22**}, \eqref{B21**}, and\eqref{G**}, and absorbing the $\epsilon$-term to LHS we have:
\begin{equation}\label{curlvttbtt}
\|\curl v_{tt}\|_{0.5}+\|\curl b_{tt}\|_{0.5}\lesssim \PP_0+\int_0^t\PP+\epsilon(\|v_{tt}\|_{1.5}+\|b_{tt}\|_{1.5}).
\end{equation}

Summing up \eqref{divvtt}, \eqref{divbtt}, \eqref{bttbdry} and \eqref{curlvttbtt}, then absorbing the $\epsilon$-term to LHS, and finally using Young's inequality and Jensen's inequality, we have
\begin{equation}
\begin{aligned}\label{divcurlvttbtt}
\|v_{tt}\|_{1.5}&\lesssim\PP_0+\int_0^t \PP~ds +\|v_{tt}^3\|_{1,\Gamma}+P(\| v\|_{2.5+\delta})\underbrace{(\|v\|_{1.5}+\|v_t\|_{1.5})}_{\lesssim\PP_0+\int_0^t\PP}\\
&\lesssim\PP_0+\int_0^t \PP~ds +\|v_{tt}^3\|_{1,\Gamma}+P(\| v\|_{2.5+\delta});\\
\|b_{tt}\|_{1.5}&\lesssim\PP_0+\int_0^t \PP~ds+P(\| v\|_{2.5+\delta},\|b\|_{2.5+\delta})\underbrace{(\|v\|_{1.5}+\|v_t\|_{1.5}+\|b_t\|_{1.5})}_{\lesssim\PP_0+\int_0^t\PP}\\
&\lesssim\PP_0+\int_0^t \PP~ds+P(\| v\|_{2.5+\delta},\|b\|_{2.5+\delta}),
\end{aligned}
\end{equation}where $\delta>0$ can be arbitratily small.

So far, we have derived all the div-curl estimates as shown in Proposition \ref{dc}. However, the control of the boundary terms containing $v$ and its time derivatives  as well as the lower order terms (i.e., $\|v\|_{2.5+\delta}$ and $\|b\|_{2.5+\delta}$ are still needed. This will be done in Section \ref{bdryofv} and Section \ref{closing}, receptively.

\section{Boundary Estimates of $v$}\label{bdryofv}
In this chapter we focus on the boundary estimates of $v^3,v_t^3,v_{tt}^3$ with the help of boundary elliptic estimates and the comparison with tangential projection. The conclusion is that 
\begin{prop}[Boundary estimates of $v,v_t,v_{tt}$]\label{bdrybdry}
\begin{equation}\label{v30}
\|v^3\|_{3,\Gamma}\lesssim \PP_0+\PP\int_0^t\PP+P(\|v\|_{2.5+\delta}).
\end{equation}

\begin{equation}\label{vt30}
\|v_t ^3\|_{2,\Gamma}\lesssim\epsilon(\|v_t\|_{2.5}+\|\TP^2 v_t\|_{0,\Gamma})+\|\TP^2(\Pi v_t)\|_{0,\Gamma}+\PP_0+\int_0^t \PP;
\end{equation}

\begin{equation}\label{vtt30}
\|v_{tt}^3\|_{1,\Gamma}\lesssim\epsilon\left(\sum_{\alpha=1}^3\|\TP v_{tt}^{\alpha}\|_{0,\Gamma}+\sum_{\alpha=1}^3\| v_{tt}^{\alpha}\|_{1}\right)+\|\TP(\Pi v_{tt})\|_{0,\Gamma}.
\end{equation}
\end{prop}
\subsection{Control of $v^3$: Boundary elliptic estimates}
From \eqref{divcurlvb00}, we still have to control $\|v^3\|_{3,\Gamma}$. Differentiating the surface tension equation $a^{\mu\alpha}N_{\mu}Q+\sigma\sqrt{g}\Delta_g \eta^{\alpha}=0$ in time and let $\alpha=3$, we have:
\begin{equation}\label{v3elliptic}
\sqrt{g}g^{ij}-\sqrt{g}g^{ij}\Gamma_{ij}^k\p_kv^3=\p_t(\sqrt{g}g^{ij})\p_i\p_i\eta^3-\p_t(\sqrt{g}g^{ij}\Gamma_{ij}^k)\p_k\eta^3-\frac{1}{\sigma}\p_t(a^{\mu 3}Q)N_{\mu},\q \text{on}\,\,\Gamma.
\end{equation}
Invoking Proposition \ref{etabdry}, we have $\|g_{ij}\|_{3,\Gamma}\leq C$ and $\|\Gamma_{ij}^k\|_{2,\Gamma}\leq C$. Therefore, by the elliptic estimates with coefficients in Sobolev spaces, one has:
\[
\|v^3\|_{3,\Gamma}\lesssim\|\p_t(\sqrt{g}g^{ij})\p_i\p_i\eta^3\|_{1,\Gamma}+\|\p_t(\sqrt{g}g^{ij}\Gamma_{ij}^k)\p_k\eta^3\|_{1,\Gamma}+\frac{1}{\sigma}\|\p_t(a^{\mu 3}Q)N_{\mu}\|_{1,\Gamma}
\]
For the term $\|\p_t(\sqrt{g}g^{ij})\p_i\p_i\eta^3\|_{1,\Gamma}$, we have:
\begin{equation}\label{v31}
\|\p_t(\sqrt{g}g^{ij})\p_i\p_i\eta^3\|_{1,\Gamma}\lesssim \|\p_t(\sqrt{g}g^{-1})\|_{1.5,\Gamma}\|\TP^2\eta^3\|_{1,\Gamma}\lesssim\|v\|_3\int_0^t\|\TP^2 v^3\|_{1,\Gamma}\lesssim\|v\|_3\int_0^t\|v\|_{3.5},
\end{equation} where we used $\cp^2\eta^3 =\int_0^t \cp^2 v^3$ since $\TP^2\eta(0)=0$.
For the term $\|\p_t(\sqrt{g}g^{ij}\Gamma_{ij}^k)\p_k\eta^3\|_{1,\Gamma}$, one expands the Christoffel symbol to get
\[
\begin{aligned}
\p_t(\sqrt{g}g^{ij}\Gamma_{ij}^k)=&-\frac{1}{\sqrt{g}}g^{mn}\p_m\eta^{\tau}\p_t\p_n\eta_{\tau}g^{ij}g^{kl}\p_l\eta^{\nu}\p_i\p_j \eta_{\nu}+\sqrt{g}\p_t(g^{ij}g^{kl}\p_l\eta^{\nu})\p_i\p_j\eta_{\nu}\\
&+\sqrt{g}g^{ij}g^{kl}\p_l\eta^{\nu}\p_i\p_j v_{\nu}.
\end{aligned}
\] 
A direct computation yields:
\begin{equation}\label{v32}
\begin{aligned}
\|\p_t(\sqrt{g}g^{ij}\Gamma_{ij}^k)\p_k\eta^3\|_{1,\Gamma}&\lesssim\sum_{k=1}^2\|v\|_{3.5}\|\p_k\eta^3\|_{1.5,\Gamma}\lesssim\|v\|_{3.5} \left(\underbrace{\p_k\eta^3(0)}_{=0}+\int_0^t \|v\|_3\right) \\
&\lesssim\|v\|_{3.5}\int_0^t\|v\|_{3}.
\end{aligned}
\end{equation}

The term containing $q$ can be easily estimated by H\"older's inequality and Sobolev embedding. Summing up \eqref{v31} and \eqref{v32} and using \eqref{datapriori} and the trace lemma, we have the boundary control of $v^3$:
\begin{equation}\label{v3bdry}
\begin{aligned}
\|v^3\|_{3,\Gamma}&\lesssim \|Q_t\|_1+P(\|v\|_{2.5+\delta},\|Q\|_{1.5,\Gamma})+P(\|v\|_{3.5})\int_0^tP(\|v\|_{3.5})\\
&\lesssim \|Q_t(0)\|_1+\int_0^t\|Q_{tt}(s)\|_1 ds+ P(\|Q_0\|_2)+\int_0^t\|Q_t(s)\|_2 ds+P(\|v\|_{2.5+\delta})+\PP\int_0^t\PP\\
&\lesssim\PP_0+\PP\int_0^t\PP+P(\|v\|_{2.5+\delta}),
\end{aligned}
\end{equation}

\subsection{Control of $v_t^3$ and $v_{tt}^3$: Comparing $\Pi X$ with $X^3$}

To control $\|v_t^3\|_{2,\Gamma}$ and $\|v_{tt}^3\|_{1,\Gamma}$, we need to use the bound for $\Pi v_t$ and $\Pi v_{tt}$. In general, we need the following argument, which was proved in Section 6.1 of \cite{disconzi2017prioriI}: 
\begin{lem}[Compare $\Pi X$ with $X\cdot N$]\label{tgproj}
For any (smooth) vector field $X$ in  $\Omega$, we have

\begin{equation} \label{XX33}
\|\TP X^3\|_{0,\Gamma}\lesssim\epsilon\left(\sum_{\alpha=1}^3\|\TP X^{\alpha}\|_{0,\Gamma}+\sum_{\alpha=1}^3\| X^{\alpha}\|_{1}\right)+\|\TP(\Pi X)\|_{0,\Gamma}.
\end{equation}
\begin{equation} \label{X3}
\begin{aligned}
\|\TP^2 X^3\|_{0,\Gamma}\lesssim&\epsilon\|X\|_{2.5}+\epsilon\|\TP^2 X\|_{0,\Gamma}+ P(\|Q\|_{1.5,\Gamma},\|\p_t X(0)\|_{0})\\
&+\|\TP^2(\Pi X)\|_{0,\Gamma}+\int_0^t P(\|\p_t X\|_0);
\end{aligned}
\end{equation}
\end{lem}
\begin{flushright}
$\square$
\end{flushright}

Let $X=v_t$ ($v_{tt}$, resp.) in \eqref{X3} (\eqref{XX33}, resp.), we have:
\begin{equation} \label{vtt3}
\|v_{tt}^3\|_{1,\Gamma}\lesssim\epsilon\left(\sum_{\alpha=1}^3\|\TP v_{tt}^{\alpha}\|_{0,\Gamma}+\sum_{\alpha=1}^3\| v_{tt}^{\alpha}\|_{1}\right)+\|\TP(\Pi v_{tt})\|_{0,\Gamma},
\end{equation}
\begin{equation} \label{vt3}
\|v_t ^3\|_{2,\Gamma}\lesssim\epsilon\|v_t\|_{2.5}+\epsilon\|\TP^2 v_t\|_{0,\Gamma}+ P(\|v_{tt}(0)\|_{0})+\|\TP^2(\Pi v_t)\|_{0,\Gamma}+\int_0^t P(\|v_{tt}\|_0).
\end{equation}Therefore we ends the proof of Proposition \ref{bdrybdry}.
\begin{flushright}
$\square$
\end{flushright}

\section{Tangential Estimates} \label{tang}
In this section we will derive the tangential estimates of $v_{ttt}, b_{ttt}$ and $v_{tt}, b_{tt}$, as well as the tangential projection $\Pi v_t$ and $\Pi v_{tt}$, which together with the div-curl estimates in Section \ref{divcurl} and the boundary estimates in Section \ref{bdryofv}, will close all the a priori estimates. Our conclusion in this section is:
\begin{prop}\label{tg}
Assume the assumptions of Lemma \ref{estimatesofa} holds, then we have:
\begin{equation}\label{tgvtttbttt}
\|v_{ttt}\|_0^2+\|b_{ttt}\|_0^2+\|\TP(\Pi v_{tt})\|_{0,\Gamma}^2\lesssim\epsilon(\|Q_{tt}\|_1^2+\|v_t\|_{2.5}^2)+P(\|v\|_{2.5+\delta})+\PP_0+\int_0^t\PP,
\end{equation}
and
\begin{equation}\label{tgvttbtt}
\|\TP v_{tt}\|_0^2+\|\TP b_{tt}\|_0^2+\|\TP^2(\Pi v_{t})\|_{0,\Gamma}^2\lesssim\epsilon\|v_t\|_{2.5}^2+P(\|v\|_{2.5+\delta})+\PP_0+\int_0^t\PP.
\end{equation}Here $\epsilon>0$ is a positive small constant and is to be determined.
\end{prop}

\textbf{Remark: }Before going to the proof, we point out thet all the boundary integrals on the fixed bottom $\Gamma_0$ vanish since we have 
\[
v^3=0,\p_i\eta^3=0
\] and thus 
\[
\p_iv^3=0, \p_t v^3=\p_t^2 v^3=\p_t^3 v^3=0, a^{31}=a^{32}=0.
\]

\subsection{Estimates of $v_{ttt},b_{ttt}$ and boundary term $\p\Pi v_{tt}$}
We start with $\|v_{ttt}\|_0^2+\|b_{ttt}\|_0^2$. From the first two equations in \eqref{MHDL}, we have
\begin{equation}\label{tgttt0}
\begin{aligned}
\frac{1}{2}\|v_{ttt}\|_0^2+\frac{1}{2}\|b_{ttt}\|_0^2&=\frac{1}{2}\|v_{ttt}(0)\|_0^2\underbrace{-\int_0^t\int_{\Omega}\p_t^3(a^{\mu\alpha}\p_{\mu}Q)\p_t^3 v_{\alpha}~dyds}_{I}+\underbrace{\int_0^t\int_{\Omega}\p_t^3(b_0^{\mu}\p_{\mu}b^{\alpha})\p_t^3v_{\alpha}~dyds}_{J} \\
&~~~~+\frac{1}{2}\|b_{ttt}(0)\|_0^2+\underbrace{\int_0^t\int_{\Omega}\p_t^3(b_0^{\mu}\p_{\mu}v^{\alpha})\p_t^3b_{\alpha}~dyds}_{K}.
\end{aligned}
\end{equation}

We observe that $J+K$ actually vanishes. Indeed, one can integrate $\p_\mu$ by parts in $J+K$ to get
\begin{equation}\label{tgtttcancel}
\begin{aligned}
J+K&=\int_0^t\int_{\Omega}b_0^{\mu}\p_t^3\p_{\mu}b^{\alpha}\p_t^3v_{\alpha}+b_0^{\mu}\p_t^3\p_{\mu}v^{\alpha}\p_t^3b_{\alpha}~dyds \\
&=\int_0^t\int_{\Omega}b_0^{\mu}\p_{\mu}(\p_t^3b^{\alpha}\p_t^3v_{\alpha})dyds\\
&=-\int_0^t\int_{\Omega}\underbrace{\p_{\mu}b_0^{\mu}}_{=0}\p_t^3b^{\alpha}\p_t^3v_{\alpha}dyds+\int_0^t\int_{\Gamma}\underbrace{b_0^{\mu}N_{\mu}}
_{=0}(\p_t^3b^{\alpha}\p_t^3v_{\alpha}) dSds=0.
\end{aligned}
\end{equation}
Therefore it suffices to control $I$. Here we remark that we have to integrate $\p_{\mu}$ by parts once the term $\p_{\mu}Q_{ttt}$ appears since there is no control of $Q_{ttt}$. After this, we invoke the fifth equation of \eqref{MHDL} to replace $Q_{ttt}|_{\Gamma}$ by the surface tension term and its time derivatives. To do this, we first expand $I$ as follows. 
\begin{equation}\label{I}
\begin{aligned}
I&=-\int_0^t\int_{\Omega}a^{\mu\alpha}\p_{\mu}Q_{ttt} \p_t^3v_{\alpha}dyds-3\int_0^t\int_{\Omega}a^{\mu\alpha}_t\p_{\mu}Q_{tt}\p_t^3v_{\alpha}dyds-3\int_0^t\int_{\Omega}a^{\mu\alpha}_{tt}\p_{\mu}Q_{t}\p_t^3v_{\alpha}dyds\\
&~~~~ -\int_0^t\int_{\Omega}a^{\mu\alpha}_{ttt}\p_{\mu}Q\p_t^3v_{\alpha}dyds\\
&=:I_1+I_2+I_3+I_4.
\end{aligned}
\end{equation}
With the help of Lemma \ref{estimatesofa} and Theorem \ref{QQ}, $I_2+I_3$ can be directly controlled by H\"older's inequality and Sobolev embedding:
\begin{equation}\label{I2I3}
\begin{aligned}
I_2+I_3&\leq\int_0^t(\|a_t\|_{2}\|\p Q_{tt}\|_{0}+\|a_{tt}\|_{0.5}\|\p Q_t\|_{1})\|v_{ttt}\|_0ds\\
&\lesssim\int_0^tP(\|v\|_3, \|v_t\|_{1.5}, \|v_{ttt}\|_0, \|Q_{tt}\|_1, \|Q_t\|_2)ds \\
&\lesssim\int_0^tP(\|v\|_3,\|v_t\|_{2.5}, \|v_{tt}\|_{1.5}, \|v_{ttt}\|_0, \|b\|_{3.5},\|b_t\|_{2.5}, \|b_{tt}\|_1)ds.
\end{aligned}
\end{equation}

For $I_1$, integrating $\p_{\mu}$ by parts, then invoking the surface tension equation, one has:
\begin{equation}\label{0I1}
\begin{aligned}
I_1&=-\int_0^t\int_{\Gamma}a^{\nu\alpha}Q_{ttt} \p_t^3 v_{\alpha}N_{\mu}dSds+\underbrace{\int_0^t\int_{\Omega}a^{\mu\alpha}Q_{ttt}\p_{\mu}\p_t^3v_{\alpha}dyds}_{I_{10}} \\
&=\underbrace{-\int_0^t\int_{\Gamma}\p_t^3\overbrace{(a^{\mu\alpha}N_{\mu}Q)}^{=-\sigma\sqrt{g}\Delta_g\eta^{\alpha}}\p_t^3v_{\alpha}dSds}_{I_{11}}+I_{10}\\
&~~~~+\underbrace{3 \int_0^t\int_{\Gamma}a^{\mu\alpha}_tN_{\mu}Q_{tt}\p_t^3v_{\alpha} dSds}_{I_{12}}+\underbrace{3\int_0^t\int_{\Gamma}a^{\mu\alpha}_{tt}N_{\mu}Q_t\p_t^3v_{\alpha}dSds}_{I_{13}}+\underbrace{\int_0^t\int_{\Gamma}\p_t^3a^{\mu\alpha}N_{\mu}Q\p_t^3v_{\alpha}dSds}_{I_{14}}.
\end{aligned}
\end{equation}

Here we can see the most cumbersome term is $I_{14}$ apart from $I_{11}$ since $a_{ttt}\in L^2(\Gamma)$ and $v_{ttt}$ cannot be controlled on the boundary. However, we can integrate $\p_{\mu}$ by parts to produce a term which cancels with $I_{14}$. We have:
\begin{equation}\label{I40}
\begin{aligned}
I_4&=-\int_0^t\int_{\Omega} \p_t^3 a^{\mu\alpha}\p_{\mu}Q \p_t^3 v_{\alpha}dyds  \\
&=-\int_0^t\int_{\Gamma}\p_t^3 a^{\mu\alpha}Q N_{\mu}\p_t^3v_{\alpha} dSds+\underbrace{\int_0^t\int_{\Omega}\p_t^3a^{\mu\alpha}Q\p_t^3 \p_{\mu}v_{\alpha} dyds}_{I_{41}} \\
&=-I_{14}+I_{41}.
\end{aligned}
\end{equation}

Up to now, it remains to control $I_{10}, I_{11}, I_{12}, I_{13}, I_{41}$. For $I_{41}$, one first differentiates $\p_t$ twice in Lemma \ref{estimatesofa} (8) to get $$\p_t^3a^{\mu\alpha}=-a^{\mu\nu}\p_{\beta}\p_t^2v_{\nu} a^{\beta\alpha}+L.O.T.$$ Therefore the main term of $I_{41}$ is 
\[
I_{41}=-\int_0^t\int_{\Omega}a^{\mu\nu}\p_{\beta}\p_t^2v_{\nu} a^{\beta\alpha}a^{\mu\alpha}Q\p_t^3 \p_{\mu}v_{\alpha} dyds+L.O.T.
\]Also we observe that 
\[
\p_t(a^{\mu\nu}\p_{\beta}\p_t^2v_{\nu} a^{\beta\alpha}a^{\mu\alpha}\p_t^2\p_{\mu}v_{\alpha})=2a^{\mu\nu}\p_{\beta}\p_t^2v_{\nu} a^{\beta\alpha}\p_t^3 \p_{\mu}v_{\alpha}+2\p_ta^{\mu\nu}\p_{\beta}\p_t^2v_{\nu} a^{\beta\alpha}\p_t^2 \p_{\mu}v_{\alpha},
\]which implies the main term of $I_{41}$ becomes
\begin{equation}
\begin{aligned}
I_{41}&=-\frac{1}{2}\int_{\Omega}a^{\mu\nu}\p_{\beta}\p_t^2v_{\nu} a^{\beta\alpha}a^{\mu\alpha}\p_t^2\p_{\mu}v_{\alpha}Q \bigg|_0^t+\frac{1}{2}\int_0^t \int_{\Omega}a^{\mu\nu}\p_{\beta}\p_t^2 v_{\nu}a^{\beta\alpha}\p_{\mu}\p_t^2 v_{\alpha}Q_t \\
&~~~~-\int_0^t\int_{\Omega}\p_t a^{\mu\nu}\p_{\beta}\p_t^2 v_{\nu}a^{\beta\alpha}\p_{\mu}\p_t^2v_{\alpha} Q +L.O.T. \\
&\lesssim P(\|v_0\|_3,\|v_{tt}(0)\|_{1},\|Q(0)\|_2,\|b_0\|_2)+\|v_{tt}\|_1^2 \|Q\|_2+\int_0^t\PP .
\end{aligned}
\end{equation}
To eliminate the term $\|v_{tt}\|_1^2 \|Q\|_2$, we first use the interpolation inequality and $\epsilon$-Young's inequality:
\[
\|v_{tt}\|_1^2 \|Q\|_2\lesssim \|v_{tt}\|_{1.5}^{4/3}\|v_{tt}\|_0^{2/3} \|Q\|_2\lesssim\epsilon \|\p_t^2 v\|_{1.5}^2+P(\|Q\|_2,\|v_{tt}\|_0)\lesssim\epsilon \|\p_t^2 v\|_{1.5}^2+P(\|Q\|_2,\|b\|_2,\|v_{tt}\|_0).
\] Then the last term can be written as the initial data plus the time integral:
\[
P(\|Q\|_2,\|b\|_2,\|v_{tt}\|_0)=P(\|Q(0)\|_2,\|b_0\|_2,\|v_{tt}(0)\|_0)+\int_0^t \PP.
\]Therefore $I_{41}$ has the following control:
\begin{equation}\label{I41}
I_{41}\lesssim P(\|v_0\|_3,\|v_{tt}(0)\|_{1},\|Q(0)\|_2,\|b_0\|_2)+\int_0^t\PP.
\end{equation}                                                                                                                                                                                                                                                                                                                                                                                                                                                                                                                                                                                                                                                                                                                                                                                                                                                                                                                                                                                                                                                                                                                                                                                                                                                                                                                                                                                                                                                                                                                                                                                                                                                                                                                                                                                                                                                                                                                                                                                                                                                                                                                                                                                                                                                                                                                                                                                                                                                                                                                                                                                                                                                                                                                                                                                                    
The control of $I_{13}$ is also straightforward if we integrate $\p_t$ by parts,:
\begin{equation}\label{I13}
\begin{aligned}
I_{13}&=3\int_{\Gamma} \p_t^2 a^{\mu\alpha}N_{\mu}Q_{t}\p_t^2v_{\alpha} ds\bigg|_{0}^t-3\int_0^t\int_{\Gamma}(\p_t^3 a^{\nu\alpha}Q_{t}+\
\p_t^2 a^{\nu\alpha}Q_{tt})N_{\mu}\p_t^2v_{\alpha}dSds.\\
&\lesssim \|\p_t^2a\|_{L^2(\Gamma)}\|Q_t\|_{L^4(\Gamma)}\|v_{tt}\|_{L^4(\Gamma)}\bigg|_0^t\\
&~~~~+\int_0^t(\|\p_t^3a\|_{L^2(\Gamma)}\|Q_t\|_{L^{\infty}(\Gamma)}+\|\p_t^2 a\|_{L^4(\Gamma)}\|Q_{tt}\|_{L^4(\Gamma)})\|v_{tt}\|_{L^2(\Gamma)}ds \\
&\lesssim \|a_{tt}\|_{0.5}\|Q_t\|_1\|v_{tt}\|_{1.5}+\int_0^t \PP \\
&\lesssim (\|v\|_2^2+\|v_t\|_{1.5})^4+\|Q_t\|_1^4 +\epsilon \|v_{tt}\|_{1.5}^2 +\int_0^t \PP \\
&\lesssim \PP_0+\int_0^t \PP+\epsilon \|v_{tt}\|_{1.5}^2.
\end{aligned}
\end{equation}Here $\epsilon>0$ need not be arbitrarily small, since we only require it can be small enough to be absorbed by $N(t)$.

The control of the remaining terms needs either to invoke the surface tension equation, or to use some tricky simplification. First, we show that the desired term $\|\TP \Pi v_{tt}\|_{0,\Gamma}$ comes from $I_{11}$. Integrating $\p_i$ by parts, one has:

\begin{equation}\label{I110}
\begin{aligned}
\frac{1}{\sigma}I_{11}&=\int_0^t\int_{\Gamma}\p_t^2\p_i\left(\sqrt{g}g^{ij}(\delta^{\alpha}_{\lambda}-g^{kl}\p_k\eta^{\alpha}\p_l\eta^{\lambda})\p_j v^{\lambda})+\sqrt{g}(g^{ij}g^{kl}-g^{lj}g^{ik})\p_j\eta^{\alpha}\p_k\eta^{\lambda}\p_l v^{\lambda}\right)\p_t^3 v_{\alpha}\\
&=\underbrace{-\int_0^t \sqrt{g}g^{ij}(\delta^{\alpha}_{\lambda}-g^{kl}\p_k\eta^{\alpha}\p_l\eta^{\lambda})\p_t^2\p_j v^{\lambda}\p_t^3\p_i v_\alpha}_{I_{111}} \\
&~~~~\underbrace{-\int_0^t\int_{\Gamma}\sqrt{g}(g^{ij}g^{kl}-g^{lj}g^{ik})\p_j\eta^{\alpha}\p_k\eta^{\lambda}\p_l\p_t^2v^{\lambda}\p_t^3\p_i v_{\alpha}}_{I_{112}}+L_{11},
\end{aligned}
\end{equation} where $L_{11}$ consists of all the terms in $\p_t^3(\sqrt{g}\Delta_g \eta^{\alpha})$ with at least one $\p_t$ falling on $\sqrt{g}(g^{ij}g^{kl}-g^{lj}g^{ik})\p_j\eta^{\alpha}\p_k\eta^{\lambda}$ and $\sqrt{g}g^{ij}(\delta^{\alpha}_{\lambda}-g^{kl}\p_k\eta^{\alpha}\p_l\eta^{\lambda})$. We only show how to control $I_{111}$ and $I_{112}$. For the control of $L$, one only needs to integrate $\p_t$ by parts. We refer readers to Section 4.1.1.3 in \cite{disconzi2017prioriI} for details. The result is
\begin{equation}\label{L11}
L_{11}\lesssim\epsilon\|v_{tt}\|_{1.5}^2+P(\|v_0\|_3,\|v_t(0)\|_{1.5},\|v_{tt}(0)\|_{1.5})+P(\|v\|_{2.5+\delta})+\int_0^t\PP.
\end{equation}

To control $I_{111}$, we recall $\Pi^{\alpha}_{\lambda}=\delta^{\alpha}_{\lambda}-g^{kl}\p_k\eta^{\alpha}\p_l\eta^{\lambda}$ to get
\[
I_{111}=-\frac{1}{2}\int_0^t\int_{\Gamma}\sqrt{g}g^{ij}\Pi^{\alpha}_{\lambda}\p_t(\p_t^2\p_jv^{\lambda}\p_t^2\p_i v_{\alpha}).
\] Integrating $\p_t$ by parts, using the symmetry of $g^{-1}$ and $\Pi$, and also $\Pi\p_t^2\p_i v=\p_i(\Pi\p_t^2v)-\p_i\Pi\p_t^2 v$, we obtain
\begin{equation}
\begin{aligned}
I_{111}&=-\frac{1}{2}\int_{\Gamma}\sqrt{g}g^{ij}\Pi^{\alpha}_{\lambda}\p_t^2\p_j v^{\lambda}\p_t^2\p_i v_{\alpha}\bigg|_0^t+\frac{1}{2}\int_0^t\int_{\Gamma}\p_t(\sqrt{g}g^{ij}\Pi_{\lambda}^{\alpha})\p_t^2\p_j v^{\lambda}\p_t^2\p_iv_{\alpha} \\
&=\underbrace{-\frac{1}{2}\int_{\Gamma}\sqrt{g}g^{ij}\p_i(\Pi^{\alpha}_{\mu}\p_t^2v_{\alpha})\p_j(\Pi^{\mu}_{\lambda}\p_t^2v^{\lambda})}_{I_{1111}}\\
&~~~~+\underbrace{\int_{\Gamma}\sqrt{g}g^{ij}\p_i\Pi^{\alpha}_{\mu}\p_t^2v_{\alpha}\p_j(\Pi^{\mu}_{\lambda}\p_t^2v^{\lambda})-\frac{1}{2}\int_{\Gamma}\p_i\Pi^{\alpha}_{\mu}\p_j\Pi^{\mu}_{\lambda}\p_t^2v_{\alpha}\p_t^2v^{\lambda}}_{I_{1112}}\\
&~~~~\underbrace{-\frac{1}{2}\int_{\Gamma}\sqrt{g}g^{ij}\Pi^{\alpha}_{\lambda}\p_t^2\p_j v^{\lambda}\p_t^2\p_i v_{\alpha}\bigg|_{t=0}}_{I_{1110}}+\underbrace{\frac{1}{2}\int_0^t\int_{\Gamma}\p_t(\sqrt{g}g^{ij}\Pi_{\lambda}^{\alpha})\p_t^2\p_j v^{\lambda}\p_t^2\p_iv_{\alpha}}_{L_{111}}.\\
\end{aligned}
\end{equation} 

The main term is $I_{1111}$. Plugging $\sqrt{g}g^{ij}=\delta^{ij}+(\sqrt{g}g^{ij}-\delta^{ij})$ and $\|\sqrt{g}g^{ij}-\delta^{ij}\|_{1.5,\Gamma}\leq\epsilon$, we get the desired term $\|\TP\Pi v_{tt}\|_{0,\Gamma}^2$ in the following way:
\begin{equation}
\begin{aligned}
I_{1111}&=-\frac{1}{2}\|\TP\Pi v_{tt}\|_{0,\Gamma}^2-\frac{1}{2}\int_{\Gamma}(\sqrt{g}g^{ij}-\delta^{ij})\p_i(\Pi^{\alpha}_{\mu}\p_t^2v_{\alpha})\p_j(\Pi^{\mu}_{\lambda}\p_t^2v^{\lambda}) \\
&\leq-\frac{1}{2}\|\TP\Pi v_{tt}\|_{0,\Gamma}^2+\epsilon\|\TP\Pi v_{tt}\|_{0,\Gamma}^2.
\end{aligned}
\end{equation}

For the remaining terms, invoking $\|\sqrt{g}g^{-1}\|_{1.5,\Gamma}\lesssim 1$ and $\|\TP\Pi\|_{1.5,\Gamma}\lesssim\|\eta\|_{3.5,\Gamma}$, one has 
\begin{equation}
\begin{aligned}
L_{111}&\lesssim\int_0^t\|\p_t(\sqrt{g}g^{-1}\Pi)\|_{L^{\infty}(\Gamma)}\|\TP v_{tt}\|_{0,\Gamma}^2\lesssim\int_0^t\|v_{tt}\|_{1.5}\|v\|_3, \\
I_{1112}&\lesssim\|\p_t(\sqrt{g}g^{-1}\Pi)\|_{L^{\infty}(\Gamma)}\|\TP\Pi\|_{L^{\infty}(\Gamma)}\|v_{tt}\|_{0,\Gamma}\|\TP\Pi v_{tt}\|_{0,\Gamma}+\|\TP\Pi\|_{L^{\infty}(\Gamma)}^2\|v_{tt}\|_{0,\Gamma}^2 \\
&\lesssim P(\|Q\|_{1.5,\Gamma})(\|v_{tt}\|_{0,\Gamma}\|\TP\Pi v_{tt}\|_{0,\Gamma}+\|v_{tt}\|_{0,\Gamma}^2)\\
&\lesssim \epsilon \|\TP\Pi v_{tt}\|_{0,\Gamma}^2 +P(\|Q\|_{1.5,\Gamma})\|v_{tt}\|_{0,\Gamma}^2 ,\\
I_{1110}&\lesssim P(\|v_{tt}(0)\|_{1,\Gamma}).
\end{aligned}
\end{equation}
Hence, we have
\begin{equation}\label{I111}
I_{111}\lesssim -\|\TP\Pi v_{tt}\|_{0,\Gamma}^2+\epsilon \|v_{tt}\|_{1.5}^2+P(\|v_{tt}(0)\|_{1.5})+P(\|Q\|_{1.5,\Gamma})+\int_0^t\PP.
\end{equation}

To end the estimates of $I_{11}$, it remains to control $I_{112}$, which requires some remarkable structures introduced in \cite{coutand2007LWP}. The detailed computation is exactly the same as Section 4.1.1.2 in \cite{disconzi2017prioriI}. The estimate for $I_{112}$ is based on the following observation: One can write 
\begin{equation}\label{I112}
I_{112}=\int_0^t\int_{\Gamma}\frac{1}{\sqrt{g}}(\p_t\det A^1+\det A^2+\det A^3),
\end{equation}where $A^1_{ij}=\p_i\eta_{\mu}\p_t^2\p_j v^{\mu}, A^2_{ij}=\p_iv_{\mu}\p_t^2\p_j v^{\mu}$ and 
\[A^3=
\begin{gathered}
\begin{pmatrix}
\p_1\eta_{\mu}\p_t^2\p_1v_{\mu}& \p_1v_{\mu}\p_t^2\p_2v_{\mu}\\
\p_2\eta_{\mu}\p_t^2\p_1v_{\mu}&\p_2v_{\mu}\p_t^2\p_2v_{\mu}
\end{pmatrix}
\end{gathered}.
\]
Now we explain how this identity holds: Consider the integrand $$(g^{ij}g^{kl}-g^{lj}g^{ik})\p_j\eta^{\alpha}\p_k\eta^{\lambda}\p_t\p_l v^{\lambda}\p_t^3\p_i v_{\alpha}.$$
 Since this vanishes if $l=i$,  we only need to consider the case when $(l,i)=(1,2)~and~(2,1)$, and then it becomes
\[
\frac{1}{\sqrt{g}}(\p_1\eta_{\mu}\p_2\eta_{\lambda}-\p_1\eta_{\lambda}\p_2\eta_{\mu})(\p_t^2\p_2 v^{\lambda}\p_t^3\p_1v^{\mu}+\p_t^3 v^{\lambda}\p_t^2\p_1v^{\mu}).
\]One the other hand, one can find
\[
\frac{1}{g}\p_t\det A^1\sim \frac{1}{g}(\p_1\eta_{\mu}\p_2\eta_{\lambda}-\p_1\eta_{\lambda}\p_2\eta_{\mu})(\p_t^2\p_2 v^{\lambda}\p_t^3\p_1v^{\mu}+\p_t^3 v^{\lambda}\p_t^2\p_1v^{\mu}),
\]while $A^2,A^3$ are present precisely to compensate the L.O.T omitted above. The result is
\begin{equation}
I_{112}\lesssim\epsilon\|\TP v_{tt}\|_{0,\Gamma}^2+\epsilon\|v_{tt}\|_{1.5}^2+P(\|v_{tt}(0)\|_{1.5})+P(\|Q\|_{1.5,\Gamma})+\int_0^t \PP.
\end{equation}

Thus, from \eqref{I110}, \eqref{I111}, \eqref{I112}, $I_{11}$ can be controlled:
\begin{equation}\label{I11}
\begin{aligned}
I_{11}&\lesssim-\|\TP v_{tt}\|_{0,\Gamma}^2+\epsilon\|v_{tt}\|_{1.5}^2+P(\|Q\|_{1.5,\Gamma},\|v\|_{2.5+\delta})\\
&~~~~+P(\|v_0\|_3,\|v_t(0)\|_{1.5},\|v_{tt}(0)\|_{1.5})+\int_0^t \PP.
\end{aligned}
\end{equation}

The remaining work is to control $I_{10},I_{12}$ via the surface tension equation. For $I_{10}$, invoking the divergence-free condition for $v$, one has 
\begin{equation}\label{I100}
\begin{aligned}
I_{10}&=-3\int_0^t\int_{\Omega}\p_t^2 a^{\mu\alpha}\p_{\mu}\p_t v_{\alpha}\p_t^3 Q-3\int_0^t\int_{\Omega}\p_t a^{\mu\alpha}\p_{\mu}\p_t^2 v_{\alpha}\p_t^3Q-\int_0^t\int_{\Omega}\p_t^3 a^{\mu\alpha}\p_{\mu} v_{\alpha}\p_t^3Q \\
&=:I_{101}+I_{102}+I_{103}
\end{aligned}
\end{equation}$I_{101}$ can be directly controlled by H\"older's inequality and Sobolev embedding after integrating $\p_t$ by parts:
\begin{equation}\label{I101}
I_{101}\lesssim\epsilon(\|Q_{tt}\|_1^2+\|v_t\|_{2.5}^2)+P(\|v\|_{2.5+\delta})+P(\|v_0\|_3,\|v_t(0)\|_{1.5},\|Q_{tt}(0)\|_1)+\int_0^t \PP.
\end{equation}
The control of $I_{103}$ is similarly as that of $I_{41}$, i.e., plugging $\p_t^3a^{\mu\alpha}=-a^{\mu\nu}\p_{\beta}\p_t^2v_{\nu} a^{\beta\alpha}+L.O.T.$ into $I_{103}$, then integrating $\p_{\beta}$by parts for the main term and integrate $\p_t$ by parts in the remainder terms. Detailed computation can be found in (4.16)-(4.18) in \cite{ignatova2016local}. The result is 
\begin{equation}\label{I103}
I_{103}\lesssim\epsilon(\|Q_{tt}\|_1^2+\|v_{tt}\|_{1.5}^2)+P(\|Q_{tt}(0)\|_1,\|Q_t(0)\|_2,\|v_{tt}(0)\|_{1.5},\|v_t(0)\|_{2.5})+P(\|v\|_{2.5+\delta})+\int_0^t\PP.
\end{equation} 

For $I_{102}$, we first integrate $\p_{\mu}$ by parts, then use Lemma \ref{estimatesofa} (8) and the surface tension equation to get
\begin{equation}\label{I1020}
\begin{aligned}
I_{102}&=-3\int_0^t\int_{\Gamma}\p_t a^{\mu\alpha} \p_t^2 v_{\alpha}\overbrace{\p_t^3Q}^{=\p_t^3 q~on~\Gamma}N_{\mu}dSds+\underbrace{3\int_0^t\int_{\Omega}\p_t a^{\mu\alpha}\p_t^2v_{\alpha}\p_{\mu}\p_t^3 Qdyds}_{L_{1021}} \\
&=3\underbrace{\int_0^t\int_{\Gamma}a^{\mu\nu}\p_{\beta}v_{\nu}a^{\beta\alpha}\p_t^2v_{\alpha}\p_t^3 Q N_{\mu}dSds}_{I_{1021}}+L_{1021}.
\end{aligned} 
\end{equation}The term $L_{1021}$ can be controlled by integrating $\p_t$ by parts. For details, we refer to (4.36) in \cite{disconzi2017prioriI}:
\begin{equation}\label{L1021}
L_{1021}\lesssim \epsilon(\|Q_{tt}\|_1^2+\|v_{tt}\|_{1.5}^2)+P(\|Q_{tt}(0)\|_1,\|v_{tt}(0)\|_{1.5},\|v_0\|_3)+P(\|v\|_2)+\int_0^t \PP.
\end{equation}

To control $I_{1021}$, we differentiate in time variable in the surface tension equation three times to get
\[
a^{\mu\nu}N_{\mu}\p_t^3 Q=-\sigma\p_t^3(\sqrt{g}\Delta_g\eta^{\nu})-3\p_t a^{\mu\nu}N_{\mu}\p_t^2 Q-3\p_t^2 a^{\mu\nu}N_{\mu}\p_t Q-\p_t^3 a^{\mu\nu}N_{\mu}Q
\] and thus 
\begin{equation}\label{I1021}
\begin{aligned}
I_{1021}&=\underbrace{-\sigma\int_0^t\int_{\Gamma}\p_{\beta}v_{\nu}a^{\beta\alpha}\p_t^2v_{\alpha}\p_t^3(\sqrt{g}\Delta_g\eta^{\nu})}_{I_{10211}}-\int_0^t\int_{\Gamma}\p_{\beta}v_{\nu}a^{\beta\alpha}\p_t^2v_{\alpha}\p_t^3 a^{\mu\nu}N_{\mu}Q  \\
&~~~~-3\int_0^t\int_{\Gamma}\p_{\beta}v_{\nu}a^{\beta\alpha}\p_t^2v_{\alpha}\p_t a^{\mu\nu}N_{\mu}\p_t^2 Q-3\int_0^t\int_{\Gamma}\p_{\beta}v_{\nu}a^{\beta\alpha}\p_t^2v_{\alpha}\p_t^2 a^{\mu\nu}N_{\mu}\p_t Q \\
\end{aligned}
\end{equation}
All the terms above can be bounded by $\int_0^t \PP$. The last 3 terms can be bounded directly by using H\"older's inequality and Sobolev embedding, whereas the control of $I_{10211}$ needs us to invoke 
\[
\p_t(\sqrt{g}\Delta_g \eta^{\nu})=\p_i\left(\sqrt{g}g^{ij}(\delta^{\alpha}_{\lambda}-g^{kl}\p_k\eta^{\alpha}\p_l\eta^{\lambda})\p_j v^{\lambda})+\sqrt{g}(g^{ij}g^{kl}-g^{lj}g^{ik})\p_j\eta^{\alpha}\p_k\eta^{\lambda}\p_l v^{\lambda}\right)
\] again.  For details, we refer to the control of $I_{21211}$ in \cite{disconzi2017prioriI}.

From \eqref{I100}, \eqref{I101}, \eqref{I103}, \eqref{I1020}, \eqref{L1021} and \eqref{I1021}, one has
\begin{equation}\label{I10}
I_{10}\lesssim\epsilon(\|Q_{tt}\|_1^2+\|v_t\|_{2.5}^2)+P(\|v\|_{2.5+\delta})+\PP_0+\int_0^t\PP
\end{equation}

The control of $I_{12}$ can be proceeded in the same way as above. We only state the basic idea and list the result. For detailed proof, we refer to Section 4.1.2.2 (control of $I_{221}$) in \cite{disconzi2017prioriI}.

To see this, invoking Lemma \ref{estimatesofa} (8) and the surface tension equation, one can re-write $I_{12}$ to be
\[
I_{12}=3\sigma\int_0^t\int_{\Gamma}\p_{\beta}v_{\nu}a^{\beta\alpha}\p_t^2(\sqrt{g}\Delta_g\eta^{\nu})\p_t^3 v_{\alpha}dSds+\cdots,
\]and then one can mimic the proof of the control of $I_{10211}$ after integrating a tangential derivative and $\p_t$ by parts. The result is 
\begin{equation}\label{I12}
I_{12}\lesssim \epsilon\|v_{tt}\|_{1.5}^2+P(\|v_0\|_3,\|v_t(0)\|_{1.5},\|v_t(0)\|_2.\|v_{tt}(0)\|_{1.5})+P(\|v\|_{2.5+\delta})+\int_0^t \PP.
\end{equation}

Plugging \eqref{I11}, \eqref{I10} and \eqref{I12} into \eqref{0I1}, we have the estimates for $I_1$
\begin{equation}\label{I1}
I_1\lesssim-\|\TP v_{tt}\|_{0,\Gamma}^2+\epsilon(\|Q_{tt}\|_1^2+\|v_t\|_{2.5}^2)+P(\|v\|_{2.5+\delta})+\PP_0+\int_0^t\PP.
\end{equation}
Then combining \eqref{0I1}, \eqref{I2I3}, \eqref{I40} and \eqref{I41}, we know $I$ satisfies the similar estimates as $I_1$. Plugging this and \eqref{tgtttcancel} into \eqref{tgttt}, we finally ends the control of $\|v_{ttt}\|_0$ and $\|b_{ttt}\|_0$ as well as $\|\TP\Pi v_{tt}\|_{0,\Gamma}$:
\begin{equation}\label{tgttt}
\|v_{ttt}\|_0^2+\|b_{ttt}\|_0^2\lesssim-\|\TP(\Pi v_{tt})\|_{0,\Gamma}^2+\epsilon(\|Q_{tt}\|_1^2+\|v_t\|_{2.5}^2)+P(\|v\|_{2.5+\delta})+\PP_0+\int_0^t\PP.
\end{equation}

\subsection{Estimates of $\p v_{tt}$ and boundary term $\p^2\Pi v_{t}$}
In this subsection, we will derive the bound \eqref{tgvttbtt}. Similarly as in the previous subsection, we first compute:
\begin{equation}\label{tgvtt0}
\begin{aligned}
\frac{1}{2}\|\TP v_{tt}\|_0^2+\frac{1}{2}\|\TP b_{tt}\|_0^2&=\frac{1}{2}\|\TP v_{tt}(0)\|_0^2\underbrace{-\int_0^t\int_{\Omega}\TP\p_t^2(a^{\nu\alpha}\p_{\mu}Q)\TP\p_t^2 v_{\alpha}~dyds}_{I^*}\\
&~~~~+\underbrace{\int_0^t\int_{\Omega}\TP\p_t^2(b_0^{\mu}\p_{\mu}b^{\alpha})\TP\p_t^2v_{\alpha}~dyds}_{J^*} \\
&~~~~+\frac{1}{2}\|\TP b_{tt}(0)\|_0^2+\underbrace{\int_0^t\int_{\Omega}\TP\p_t^2(b_0^{\mu}\p_{\mu}v^{\alpha})\TP\p_t^2b_{\alpha}~dyds}_{K^*}.
\end{aligned}
\end{equation}
We observe that the highest order term in $J^*$ cancels with that in $K^*$. Indeed, one can integrate $\p_\mu$ by parts in $J^*+K^*$ to get
\begin{equation}\label{tgvttcancel}
\begin{aligned}
J^*+K^*&=\int_0^t\int_{\Omega}b_0^{\mu}\TP\p_t^2\p_{\mu}b^{\alpha}\TP\p_t^2v_{\alpha}+b_0^{\mu}\TP\p_t^2\p_{\mu}v^{\alpha}\TP\p_t^2b_{\alpha}~dyds \\
&~~~~+\int_0^t\int_{\Omega}\TP\p_t^2 b^{\alpha}\TP b_0^{\mu} \p_{\mu}\p_t^2v_{\alpha} dyds +\int_0^t\int_{\Omega}\TP\p_t^2v^{\alpha}\TP b_0^{\mu}\p_t^2\p_{\mu}b_{\alpha}dyds\\
&\leq\int_0^t\int_{\Omega}b_0^{\mu}\p_{\mu}(\TP\p_t^2b^{\alpha}\TP\p_t^2v_{\alpha})dyds+\int_0^t \|b_{tt}\|_1 \|v_{tt}\|_1\|b_0\|_3 ds\\
&=-\int_0^t\int_{\Omega}\underbrace{\p_{\mu}b_0^{\mu}}_{=0}\TP\p_t^2b^{\alpha}\TP\p_t^2v_{\alpha}dyds+\int_0^t\int_{\Gamma}\underbrace{b_0^{\mu}N_{\mu}}
_{=0}(\TP\p_t^2b^{\alpha}\TP\p_t^2v_{\alpha}) dSds\\
&~~~~+\int_0^t \|b_{tt}\|_1 \|v_{tt}\|_1\|b_0\|_3 ds\\
&\lesssim \int_0^t\PP.
\end{aligned}
\end{equation}
Therefore, it suffices to control $I^*$. In face, the only highest order term in $I^*$ is $\TP\p_t^2v_{\alpha}a^{\mu\alpha}\TP\p_t^2\p_{\mu}Q$ which can be controlled by integrating $\p_{\mu}$ by parts and then invoking the surface tension equation again, while the others can be controlled by H\"older's inequality and Sobolev embedding. 
\begin{equation}\label{I*0}
\begin{aligned}
I^*&=-\int_0^t\int_{\Omega}\TP\p_t^2v_{\alpha}a^{\mu\alpha}\TP\p_t^2\p_{\mu}Q dyds\underbrace{-\int_0^t\int_{\Omega}\TP\p_t^2v_{\alpha}r^{\alpha}dyds}_{I_0^*} \\
&=-\int_0^t\int_{\Gamma}\TP\p_t^2 v^{\alpha}a^{\mu\alpha}N_{\mu} \TP\p_t^2Q dSds+\int_0^t\int_{\Omega}\TP\p_{\mu}\p_t^2v_{\alpha}a^{\mu\alpha}\TP\p_t^2Q dyds +I_0^*\\
&=:I_1^*+I_2^*+I_0^*,
\end{aligned}
\end{equation}where in $I_0^*$ we have $$-r^{\alpha}=\TP a^{\mu\alpha} \p_{\mu}Q_{tt}+a^{\mu\alpha}_t \TP\p_{\mu}Q_t+\TP a^{\mu\alpha}_t \p_{\mu}Q_t+a^{\mu\alpha} \p_{\mu}Q+a^{\mu\alpha}_{tt} \TP\p_{\mu}Q+\TP^2a^{\mu\alpha}_t \p_{\mu}Q,$$ and thus we have the control for $I_0^*$:
\begin{equation}\label{I0*}
I_0^*\lesssim\int_0^t P(\|v\|_3,\|v_t\|_2,\|v_{tt}\|_{1.5},\|Q\|_2,\|Q_t\|_2, \|Q_{tt}\|_1) ds\lesssim\int_0^t\PP.
\end{equation}

$I_2^*$ can also be directly controlled by H\"older's inequality and Sobolev embedding. First the divergence-free condition for $v$ implies $\TP\p_t^2(a^{\mu\alpha}\p_{\mu}v_{\alpha})=0$. Expanding this and plugging it into $I_2^*$, one can get 
\begin{equation}\label{I2*}
I_2^*\lesssim\int_0^t P(\|v\|_3,\|v_t\|_2,\|v_{tt}\|_{1.5},\|Q_{tt}\|_1) ds\lesssim \int_0^t\PP.
\end{equation}

Now it remains to control $I_1^*$. invoking the surface tension equation again, we have
\begin{equation}\label{I1*0}
I_{1}^*=-\int_0^t\int_{\Gamma}\TP\p_t^2(a^{\mu\alpha }Q)+\cdots=\sigma\underbrace{\int_0^t\int_{\Gamma}\TP\p_t^2(\sqrt{g}\Delta_g \eta^{\alpha})\TP\p_t^2v_{\alpha}}_{I_{11}^*}+\cdots,
\end{equation}where the omitted terms can be directly bounded by $\int_0^t \PP.$ 
For $I_{11}^*$, one can mimic the proof of controlling $I_{11}$: taking derivative $\TP\p_t$ in the identity 
\[
\p_t(\sqrt{g}\Delta_g \eta^{\nu})=\p_i\left(\sqrt{g}g^{ij}(\delta^{\alpha}_{\lambda}-g^{kl}\p_k\eta^{\alpha}\p_l\eta^{\lambda})\p_j v^{\lambda})+\sqrt{g}(g^{ij}g^{kl}-g^{lj}g^{ik})\p_j\eta^{\alpha}\p_k\eta^{\lambda}\p_l v^{\lambda}\right),
\] and plugging that into $I_{11}^*$, one gets
\begin{equation}\label{I11*}
\begin{aligned}
I_{11}^*&=-\int_0^t\int_{\Gamma}\sqrt{g}g^{ij}(\delta^{\alpha}-g^{kl}\p_{k}\eta^{\alpha}\p_l\eta_{\lambda})\p_t\TP\p_j v^{\lambda}\p_t^2\TP\p_i v_{\alpha}\\
&~~~-\int_0^t\int_{\Gamma}\sqrt{g}(g^{ij}g^{kl}-g^{lj}g^{ik})\p_j\eta^{\alpha}\p_k\eta^{\lambda}\p_l\p_t\TP v^{\lambda}\p_t^2\TP\p_i v_{\alpha}+L_{111}^*\\
&=:I_{111}^*+I_{112}^*+L_{11}^*,
\end{aligned}
\end{equation} where $L_{11}^*$ is the analogue of $L_{11}$. The term $I_{112}^*$ is the analogue of $I_{112}$ which requires using the tricky determinant computation, and is estimated in the same way as $I_{112}$.
\begin{equation}\label{L11*}
I_{112}\lesssim\int_0^t\PP,~~~L_{11}^*\lesssim\epsilon\|v_t\|_{2.5}^2+P(\|v\|_{2.5+\delta})+P(\|v_0\|_3)+\int_0^t\PP
\end{equation}
For $I_{111}^*$, we just need to replace all the $\p_t^2$ appearing when controlling $I_{111}$ by $\TP\p_t$, and repeat all the steps, to get
\begin{equation}\label{I111*}
I_{111}^*\lesssim-\|\TP^2\Pi v_t\|_{0,\Gamma}^2+\epsilon\|v_t\|_{2.5}^2+P(\|v\|_{2.5+\delta})+P(\|v_0\|_3,\|Q(0)\|_2)+\int_0^t\PP.
\end{equation}
Combining and \eqref{I1*0}, \eqref{I11*}, \eqref{I111*} and \eqref{L11*} , we have
\begin{equation}\label{I1*}
I_1^*\lesssim-\|\TP^2\Pi v_t\|_{0,\Gamma}^2+\epsilon\|v_t\|_{2.5}^2+P(\|v\|_{2.5+\delta})+P(\|v_0\|_3,\|Q(0)\|_2)+\int_0^t\PP.
\end{equation}

Plugging \eqref{I0*}, \eqref{I2*} and \eqref{I1*} into \eqref{I*0}, we can derive the desired estimates from \eqref{tgvtt0} and \eqref{tgvttcancel}:
\begin{equation}\label{tgvtt}
\|\TP v_{tt}\|_0^2+\|\TP b_{tt}\|_0^2\lesssim -\|\TP^2\Pi v_t\|_{0,\Gamma}^2+\epsilon\|v_t\|_{2.5}^2+P(\|v\|_{2.5+\delta})+\PP_0+\int_0^t\PP.
\end{equation}

\section{Closing the estimates}\label{closing}
In this section we are going to close all the a priori estimates.

\subsection{Estimates at $t=0$}

Before summarising all the estimates we have gotten, we point out that, so far, all the estimates contain the initial value of several quantities. In this section we will control all these quantities in terms of the initial data, i.e. $v_0$ and $b_0.$ It is exactly here that we require the a priori estimates depend on $\|v_0\|_{4,\Gamma}$. Assume we have the a priori bound for $v_0$ and $b_0$, then the control of $\|b_t(0)\|_{2.5}$ automatically holds
\begin{equation}\label{btdata}
\|b_t(0)\|_{2.5}=\|(b_0\cdot\p)v_0\|_{2.5}\lesssim\|b_0\|_{2.5}\|v_0\|_{3.5}.
\end{equation} While the control of $\|v_t(0)\|_{2.5}$ requires the a priori bound for $Q_0:=Q(0)$. Our basic idea to proceed the remaining steps is:
\begin{equation}\label{data}
\underbrace{v_0,b_0}_{\Rightarrow b_t(0)}\xRightarrow{\Delta}Q_0\xRightarrow{\eqref{MHDL}}v_t(0)
\left.
\begin{cases}
&\xRightarrow{\p_t \eqref{MHDL}}b_{tt}(0)\\
&\xRightarrow{\Delta}Q_t(0)
\end{cases}\right\}\xRightarrow{\p_t\eqref{MHDL}} v_{tt}(0)
\left.
\begin{cases}
&\xRightarrow{\p_t^2\eqref{MHDL}}b_{ttt}(0)\\
&\xRightarrow{\Delta}Q_{tt}(0)
\end{cases}\right\}\xRightarrow{\p_t^2\eqref{MHDL}} v_{ttt}(0),
\end{equation}
here `$\Delta$' means using elliptic estimates as in Chapter \ref{ellipticQ}, `$\p_t$' means differentiating the MHD equation with respect to time variable $t$.

The first step is to control $\|Q_0\|_{3.5}$. Since $\eta(0)=$ Id, we can derive the estimate for $\|Q_0\|_{3.5}$ from the original MHD system \eqref{MHD}. Taking divergence in the first equation $D_t u-(B\cdot \nabla)B=-\nabla P$ and set $t=0$, one has
\[
\begin{aligned}
-\Delta Q_0=[\nabla, D_t]v|_{t=0} -[\nabla, b_0\cdot\nabla]b_0=\p_{\mu}v^{\nu}_0\p_{\mu}v^{\nu}_0-\p_{\mu}b^{\nu}_0\p_{\nu}b^{\mu}_0&~~in~~\Omega \\
\frac{\p Q_0}{\p N}=0&~~on~~\Gamma_0 \\
Q_0=0&~~on~~\Gamma
\end{aligned}
\] Then the standard elliptic estimate yields that
\begin{equation}\label{Qdata}
\|Q_0\|_{3.5}\lesssim\|v_0\|_{2.5}\|v_0\|_3+\|b_0\|_{2.5}\|b_0\|_3,
\end{equation} and thus one can derive the bound for $\|v_t(0)\|_{2.5}$ as well as $\|b_{tt}\|_{1.5}$:
\begin{equation}\label{vtbtdata}
\begin{aligned}
\|v_t(0)\|_{2.5}&\leq \|b_0\cdot\p b_0\|_{2.5}+\|\p Q_0\|_{2.5}\lesssim \|b_0\|_{2.5}\|b_0\|_{3.5}+\|Q_0\|_{3.5};\\
\|b_{tt}(0)\|_{1.5}&=\|b_0\cdot\p v_t\|_{1.5}\lesssim\|b_0\|_2\|v_t\|_{2.5}\lesssim \|b_0\|_2(\|b_0\|_{2.5}\|b_0\|_{3.5}+\|Q_0\|_{3.5}).
\end{aligned}
\end{equation}

To derive the bound for $\|Q_t(0)\|_{2.5}$, one needs to invoke \eqref{Qtin} and restrict it at $t=0$, with the following boundary condition
\[
\begin{aligned}
\frac{\p Q_t(0)}{\p N}=g^*|_{t=0}~~&\text{ on }\Gamma_0\text{ as in }\eqref{Qtbd} \\
Q_t(0)=q_t(0)=-\p_ta^{\mu\alpha}N_{\mu}q|_{t=0}-\sigma \Delta v_0^3~~&\text{ on }\Gamma,
\end{aligned}
\]and the standard elliptic estimate yields that 
\begin{equation}\label{Qtdata}
\|Q_t(0)\|_{2.5}\lesssim P(\|v_0\|_{3.5},\|b_0\|_{3.5},\|v_0\|_{4,\Gamma}),
\end{equation} and thus one can derive the bound for $\|v_{tt}(0)\|_{2.5}$ as well as $\|b_{ttt}\|_{0}$ by time differentiating \eqref{MHDL} again:
\begin{equation}\label{vtbtdata'}
\begin{aligned}
\|v_{tt}(0)\|_{1.5}&\leq \|b_0\cdot\p b_t(0)\|_{1.5}+\|\p_t (a^{\mu\alpha} \p_{\mu}Q)|_{t=0}\|_{1.5}\lesssim P(\|v_0\|_{3.5}\|b_0\|_{3.5},\|v_0\|_{4,\Gamma});\\
\|b_{ttt}(0)\|_{0}&=\|b_0\cdot\p v_{tt}\|_0\lesssim\|b_0\|_2\|v_{tt}\|_0\lesssim P(\|v_0\|_{3.5}\|b_0\|_{3.5},\|v_0\|_{4,\Gamma}).
\end{aligned}
\end{equation}

We remark that the last estimate illustrates that the term $\|v_0\|_{4,\Gamma}$ is necessary in the a priori estimates due to $\Delta v_0^3$ on the boundary. Besides, one can continue the steps by following the idea in \eqref{data} to get the bound for $\|v_{ttt}(0)\|_{0}$ and $\|Q_{tt}\|_1$ so we omit the details. We conclude that
\begin{equation}\label{datapriori}
\left.
\begin{aligned}
&~~~~\|v_t(0)\|_{2.5}+\|v_{tt}(0)\|_{1.5}+\|v_{ttt}(0)\|_0\\
&+\|b_t(0)\|_{2.5}+\|b_{tt}(0)\|_{1.5}+\|b_{ttt}(0)\|_0 \\
&+\|Q(0)\|_{3.5}+\|Q_{t}(0)\|_{2.5}+\|Q_{tt}(0)\|_1 \\
\end{aligned}\right\}\lesssim P(\|v_0\|_{3.5}\|b_0\|_{3.5},\|v_0\|_{4,\Gamma}).
\end{equation}

\subsection{Rewrite and summarise the estimates}
Now we summarise all the estimates that we have gotten. In order to apply Gronwall-type inequality, we have to ensure all of the a priori quantities are controlled by the sum of the initial data and the time integral of these quantities. Therefore we also need to rewrite the estimates of pressure shown in Proposition \ref{QQ}.

\paragraph*{Estimates of $\eta$:} 
\begin{equation}\label{etaclose}
\|\eta\|_{3.5}\leq\|v_0\|_{3.5}+\int_0^t\|v(s)\|_{3.5}ds,
\end{equation} obviously holds.

\paragraph*{Estimate of $v$:} 

From \eqref{estimatesofQtt}, \eqref{datapriori}, \eqref{divcurlvb} and \eqref{v30} in Proposition \ref{v30}, we have:
\begin{equation}\label{vclose0}
\begin{aligned}
\|v\|_{3.5}^2&\lesssim P(\|v_0\|_{3.5},\|b_0\|_{3.5})+\PP \int_0^t\PP +\|Q_t\|_1^2+P(\|v\|_{2.5+\delta})\\
&\lesssim P(\|v_0\|_{3.5},\|b_0\|_{3.5})+\PP \int_0^t\PP +\|Q_t(0)\|_1^2+\int_0^t \|Q_{tt}(s)\|_1^2 ds+P(\|v\|_{2.5+\delta}) \\
&\lesssim P(\|v_0\|_{3.5},\|b_0\|_{3.5})+\PP \int_0^t\PP+P(\|v\|_{2.5+\delta})
\end{aligned}
\end{equation}

\paragraph*{Estimate of $v_t$:} 

For $v_t$, we notice that the $\epsilon$-term on the RHS of \eqref{vt3} can be absorbed by $\|v_t\|_{2.5}$. Therefore, combining this with \eqref{divcurlvtbt}, \eqref{tgvttbtt}, we get
\begin{equation}\label{vtclose0}
\begin{aligned}
\|v_{t}\|_{2.5}^2&\lesssim\PP_0+\int_0^t\PP+\|v_t\|_{2,\Gamma}^2+\|\TP v_{tt}\|_0^2+\|\TP b_{tt}\|_0^2 \\
&\lesssim \PP_0+\int_0^t\PP+P(\|v\|_{2.5+\delta})+\underbrace{\|\TP^2\Pi v_t\|_{0,\Gamma}^2+\|\TP v_{tt}\|_0^2+\|\TP b_{tt}\|_0^2}_{\text{using }\eqref{tgvttbtt}} \\
&\lesssim  P(\|v_0\|_{3.5},\|b_0\|_{3.5},\|v_0\|_{4,\Gamma})+\int_0^t\PP+P(\|v\|_{2.5+\delta}).
\end{aligned}
\end{equation}

\paragraph*{Estimates of $v_{tt},v_{ttt},b_{ttt}$:} 

Similarly, one can get the estimates of $v_{tt}$, $v_{ttt}$ and $b_{ttt}$ simultaneously just by mimicing the derivation of \eqref{vtclose0}. Using \eqref{divcurlvttbtt}, \eqref{tgvtttbttt} and absorbing all the $\epsilon$-terms to LHS, we have
\begin{equation}\label{vttclose0}
\begin{aligned}
\|v_{tt}\|_{1.5}^2+\| v_{ttt}\|_0^2+\| b_{ttt}\|_0^2&\lesssim\PP_0+\int_0^t \PP +\|v_{tt}^3\|_{1,\Gamma}^2 +\|\TP v_{ttt}\|_0^2+\|\TP b_{ttt}\|_0^2 \\
&\lesssim\PP_0+\int_0^t \PP +\underbrace{\|\TP(\Pi v_{tt})\|_{0,\Gamma}^2+\|\TP v_{ttt}\|_0^2+\|\TP b_{ttt}\|_0^2}_{\text{using }\eqref{tgvtttbttt}}\\
&\lesssim P(\|v_0\|_{3.5},\|b_0\|_{3.5},\|v_0\|_{4,\Gamma})+\int_0^t\PP\\
&~~~~+P(\|v\|_{2.5+\delta})+\epsilon\|Q_{tt}\|_1^2.
\end{aligned}
\end{equation}

\paragraph*{Estimates of $b,b_t,b_{tt}$:} These estimates have been derived in Proposition \ref{dc}:
\begin{equation}\label{bbtbtt}
\|b\|_{3.5}^2+\|b_t\|_{2.5}^2+\|b_{tt}\|_{1.5}^2\lesssim P(\|v\|_{2.5
+\delta},\|b\|_{2.5+\delta})+P(\|v_0\|_{3.5},\|b_0\|_{3.5},\|v_0\|_{4,\Gamma})+\int_0^t\PP.
\end{equation}

Since our a priori quantities contain $Q,Q_t,Q_{tt}$, we still need to rewrite the pressure estimates into the sum of initial data and time integral of these a priori quantites instead of only a polynomial of these quantities as shown in Proposition \ref{ellipticQ}.
\paragraph*{Estimates of $Q$, $Q_t$, $Q_{tt}$:}
For the estimates of $Q$, we invoke \eqref{vtclose0} and \eqref{bbtbtt} to rewrite the pressure estimates in Proposition \ref{QQ} as follows
\begin{equation}\label{Qclose}
\begin{aligned}
\|Q\|_{3.5}^2&\lesssim P(\|v\|_{2.5+\delta},\|b\|_{2.5+\delta})+P(\|v_0\|_{3.5},\|b_0\|_{3.5},\|v_0\|_{4,\Gamma})+\int_0^t\PP+\|b_0\|_{2.5}^2(\PP_0+\int_0^t\PP)\\
&\lesssim P(\|v\|_{2.5+\delta},\|b\|_{2.5+\delta})+P(\|v_0\|_{3.5},\|b_0\|_{3.5},\|v_0\|_{4,\Gamma})+\int_0^t\PP.
\end{aligned}
\end{equation}
Similary as above, we rewrite $\|v_t\|_{1.5}, \|v\|_{2.5}, \|v\|_2, \|Q\|_{2.5},\|Q_t\|_1, \|b_t\|_{1.5},\|b\|_2$ as the sum of initial data and time integral of the a priori quantities, then use Young's inequality and Jensen's inequality and invoke \eqref{vtclose0}, \eqref{vttclose0} and \eqref{bbtbtt} to get
\begin{equation}\label{Qtclose}
\|Q_t\|_{2.5}^2\lesssim  P(\|v\|_{2.5+\delta})+P(\|b\|_{2.5+\delta})+P(\|v_0\|_{3.5},\|b_0\|_{3.5},\|v_0\|_{4,\Gamma})+\int_0^t\PP \\
\end{equation}
and
\begin{equation}\label{Qttclose}
\|Q_{tt}\|_{1}^2\lesssim P(\|v\|_{2.5+\delta})+P(\|b\|_{2.5+\delta})+P(\|v_0\|_{3.5},\|b_0\|_{3.5},\|v_0\|_{4,\Gamma})+\int_0^t\PP. \\
\end{equation}

\subsection{Eliminate lower order terms}
So far, it remains to deal with the lower order terms containing neither in the time integral, nor in the initial data, specifically, $P(\|v\|_{2.5+\delta})$ and $P(\|b\|_{2.5+\delta})$ for arbitrarily small $\delta\in(0,0.5)$. Therefore, it suffices to choose a suitable $\delta\in(0,0.5)$ and control $P(\|v\|_{2.5+\delta})$.

\paragraph*{Control of $P(\|v\|_{2.5+\delta})$:} Since $\|v\|_{2.5+\delta}\leq\frac{1}{2}+\frac{1}{2}\|v\|^2_{2.5+\delta}$, we may assume  $P(\|v\|_{2.5+\delta})$ is the combination of terms of the form $\|v\|_{2.5+\delta}^d$ with $d\geq 2$. Then by the interpolation inequality in Lemma \ref{gnsineq}, we have
\[
\|v\|_{2.5+\delta}^d\lesssim \|v\|_{3}^{2\delta d}\|v\|_0^{(1-2\delta)d}.
\] Then choose $\delta$ sufficiently close to 0, for different $d$'s, such that 
\[
p_d:=\frac{1}{d\delta}>1.
\]One can use $\epsilon$-Young's inequality with $p_d$ and its dual index to derive 
\[
\|v\|_{2.5+\delta}^k\lesssim \epsilon\|v\|_{3}^2+\|v\|_0^{b}\lesssim\epsilon\|v\|_{3.5}^2+P(\|v_0\|_{2.5})+\int_0^tP(\|v_t(s)\|_{2.5})ds\text{ for some }b>0,
\] and thus 
\begin{equation}\label{errorv}
P(\|v\|_{2.5+\delta})\lesssim\epsilon\|v\|_{3.5}^2+P(\|v_0\|_{2.5})+\int_0^t \PP.
\end{equation}
Similarly we have
\begin{equation}\label{errorb}
P(\|b\|_{2.5+\delta})\lesssim\epsilon\|b\|_{3.5}^2+P(\|b_0\|_{2.5})+\int_0^t \PP.
\end{equation}
\subsection{Gronwall-type argument}
Recall in \eqref{N(t)} we have 
\[
\begin{aligned}
N(t)=\|\eta\|_{3.5}^2&+\|v\|_{3.5}^2+\|v_t\|_{2.5}^2+\|v_{tt}\|_{1.5}^2+\|v_{ttt}\|_{0}^2+\|b\|_{3.5}^2+\|b_t\|_{2.5}^2+\|b_{tt}\|_{1.5}^2+\|b_{ttt}\|_{0}^2\\
&+\|Q\|_{3.5}^2+\|Q_t\|_{2.5}^2+\|Q_{tt}\|_1^2.
\end{aligned}
\]
Combining this with \eqref{datapriori}-\eqref{errorb}, and absorbing all the $\epsilon$-terms, we have proved:
\[
N(t)\lesssim P(\|v_0\|_{3.5},\|b_0\|_{3.5},\|v_0\|_{4,\Gamma})+P(N(t))\int_0^t P(N(s))ds.
\]

By the Gronwall-type argument in \cite{tao2006nonlinear}, we have:
\[
N(t)\leq C(\|v_0\|_{3.5},\|b_0\|_{3.5},\|v_0\|_{4,\Gamma}),
\]as desired. This ends the proof of our result.
\begin{flushright}
$\square$
\end{flushright}


\begin{thebibliography}{}

\bibitem[\protect\astroncite{Alazard et~al.}{2014}]{alazard2014cauchy}
Alazard, T., Burq, N., and Zuily, C. (2014).
\newblock On the cauchy problem for gravity water waves.
\newblock {\em Inventiones mathematicae}, 198(1): 71--163.

\bibitem[\protect\astroncite{Berninger}{2007}]{H-0.5trace}
Berninger, H. (2007).
\newblock Domain decomposition methods for elliptic problems with jumping
  nonlinearities and application to the richards equation.
\newblock {\em PhD Thesis}.

\bibitem[\protect\astroncite{Brezis}{2018}]{sobolevinterpolation}
Brezis, H. and Mironescu, P. (2018).
\newblock Gagliardo-Nirenberg inequalities and non-inequalities:  the full story.
\newblock {\em Annales de l'Institut Henri Poincar\'e (C) Non Linear Analysis},
35(5): 1355-1376.

\bibitem[\protect\astroncite{Chen and Ding}{2019}]{chendingMHDST}
Chen, P., Ding, S. (2019).
\newblock Inviscid Limit for the Free-Boundary problems of MHD Equations with or without Surface Tension.
\newblock {\em arXiv: 1905.13047}, preprint.

\bibitem[\protect\astroncite{Christodoulou and
  Lindblad}{2000}]{christodoulou2000motion}
Christodoulou, D. and Lindblad, H. (2000).
\newblock On the motion of the free surface of a liquid.
\newblock {\em Communications on Pure and Applied Mathematics},
  53(12): 1536--1602.

\bibitem[\protect\astroncite{Coutand and Shkoller}{2007}]{coutand2007LWP}
Coutand, D. and Shkoller, S. (2007).
\newblock Well-posedness of the free-surface incompressible euler equations
  with or without surface tension.
\newblock {\em Journal of the American Mathematical Society}, 20(3): 829--930.

\bibitem[\protect\astroncite{Disconzi and Kukavica}{2017}]{disconzi2017prioriI}
Disconzi, M.~M. and Kukavica, I. (2019).
\newblock A priori estimates for the free-boundary Euler equations with surface
  tension in three dimensions.
\newblock {\em Nonlinearity}, Vol 32(9), 3369-3405.

\bibitem[\protect\astroncite{Disconzi et~al.}{2018}]{DKT}
Disconzi, M.~M., Kukavica, I., and Tuffaha, A. (2019).
\newblock A Lagrangian interior regularity result for the incompressible free
  boundary Euler equation with surface tension.
\newblock {\em SIAM Journal of Mathematical Analysis}, Vol 51(5), 3982-4022.

\bibitem[\protect\astroncite{Dong and Kim}{2010}]{donghongjieBMO}
Dong, H. and Kim, D. (2010).
\newblock Elliptic equations in divergence form with partially BMO
  coefficients.
\newblock {\em Archive for Rational Mechanics and Analysis}, 196(1): 25--70.

\bibitem[\protect\astroncite{Ebin}{1987}]{ebin1987equations}
Ebin, D.~G. (1987).
\newblock The equations of motion of a perfect fluid with free boundary are not
  well posed.
\newblock {\em Communications in Partial Differential Equations},
  12(10): 1175--1201.

\bibitem[\protect\astroncite{Gu and Wang}{2019}]{gu2016construction}
Gu, X. and Wang, Y. (2019).
\newblock On the construction of solutions to the free-surface incompressible
  ideal magnetohydrodynamic equations.
\newblock {\em Journal de Math\'ematiques Pures et Appliqu\'ees}, Vol. 128: 1-41.

\bibitem[\protect\astroncite{Guo}{2019}]{GuoMHDSTviscous}
Guo, B., Zeng, L. and Ni, G. (2019).
\newblock Decay rates for the Viscous Incompressible MHD with and without Surface Tension.
\newblock {\em Computers \& Mathematics with Applications}, 77(12): 3224-3249.

\bibitem[\protect\astroncite{Hao}{2017}]{hao2017motion}
Hao, C. (2017).
\newblock On the Motion of Free Interface in Ideal Incompressible MHD.
\newblock {\em Archive for Rational Mechanics and Analysis}, Vol. 224(2): 515-553.

\bibitem[\protect\astroncite{Hao and Luo}{2014}]{hao2014priori}
Hao, C. and Luo, T. (2014).
\newblock A priori estimates for free boundary problem of incompressible
  inviscid magnetohydrodynamic flows.
\newblock {\em Archive for Rational Mechanics and Analysis}, 212(3): 805--847.

\bibitem[\protect\astroncite{Hao and Luo}{2018}]{hao2018ill}
Hao, C. and Luo, T. (2020).
\newblock Ill-posedness of free boundary problem of the incompressible ideal
  MHD.
\newblock {\em Communications in Mathematical Physics}, Vol. 376, 259-286.

\bibitem[\protect\astroncite{Ignatova and Kukavica}{2016}]{ignatova2016local}
Ignatova, M. and Kukavica, I. (2016).
\newblock On the local existence of the free-surface Euler equation with
  surface tension.
\newblock {\em Asymptotic Analysis}, 100(1-2): 63--86.

\bibitem[\protect\astroncite{Kato and Ponce}{1988}]{kato1988commutator}
Kato, T. and Ponce, G. (1988).
\newblock Commutator estimates and the Euler and Navier-Stokes equations.
\newblock {\em Communications on Pure and Applied Mathematics}, 41(7): 891--907.

\bibitem[\protect\astroncite{Kukavica et~al.}{2017}]{kukavica2017local}
Kukavica, I., Tuffaha, A., and Vicol, V. (2017).
\newblock On the local existence and uniqueness for the 3D Euler equation with
  a free interface.
\newblock {\em Applied Mathematics \& Optimization}, 76(3): 535--563.

\bibitem[\protect\astroncite{Li}{2019}]{lidong2019commutator}
Li, D. (2019).
\newblock On Kato-Ponce and fractional Leibniz.
\newblock {\em Revista Matemática Iberoamericana}, 35(Issue 1): 23--100.

\bibitem[\protect\astroncite{Lindblad}{2002}]{lindblad2002}
Lindblad, H. (2002).
\newblock Well-posedness for the linearized motion of an incompressible liquid
  with free surface boundary.
\newblock {\em Communications on Pure and Applied Mathematics},
  56(02): 153--197.

\bibitem[\protect\astroncite{Lindblad}{2005}]{lindblad2005well}
Lindblad, H. (2005).
\newblock Well-posedness for the motion of an incompressible liquid with free
  surface boundary.
\newblock {\em Annals of mathematics}, pages 109--194.

\bibitem[\protect\astroncite{Luo and Zhang}{2019}]{luozhang2019MHD2.5}
Luo, C. and Zhang, J. (2020).
\newblock A regularity result for the incompressible magnetohydrodynamics equations
  with free surface boundary.
\newblock {\em Nonlinearity}, Vol. 33(4), 1499-1527.

\bibitem[\protect\astroncite{Schweizer}{2005}]{SchweizerFreeEuler}
Schweizer, B. (2005).
\newblock On the three-dimensional {E}uler equations with a free boundary
  subject to surface tension.
\newblock {\em Ann. Inst. H. Poincar\'e Anal. Non Lin\'eaire}, 22(6): 753--781.

\bibitem[\protect\astroncite{Secchi and Trakhinin}{2013}]{secchi2013well}
Secchi, P. and Trakhinin, Y. (2013).
\newblock Well-posedness of the plasma--vacuum interface problem.
\newblock {\em Nonlinearity}, 27(1): 105-169.

\bibitem[\protect\astroncite{Shatah and Zeng}{2008a}]{shatah2008geometry}
Shatah, J. and Zeng, C. (2008a).
\newblock Geometry and a priori estimates for free boundary problems of the
  euler's equation.
\newblock {\em Communications on Pure and Applied Mathematics}, 61(5): 698--744.

\bibitem[\protect\astroncite{Shatah and Zeng}{2008b}]{shatah2008priori}
Shatah, J. and Zeng, C. (2008b).
\newblock A priori estimates for fluid interface problems.
\newblock {\em Communications on Pure and Applied Mathematics}, 61(6): 848--876.

\bibitem[\protect\astroncite{Shatah and Zeng}{2011}]{shatah2011local}
Shatah, J. and Zeng, C. (2011).
\newblock Local well-posedness for fluid interface problems.
\newblock {\em Archive for Rational Mechanics and Analysis}, 199(2): 653--705.

\bibitem[\protect\astroncite{Sun et~al.}{2017}]{sun2017well}
Sun, Y., Wang, W., and Zhang, Z. (2019).
\newblock Well-posedness of the plasma-vacuum interface problem for ideal
  incompressible MHD.
\newblock {\em Archive for Rational Mechanics and Analysis}, Vol. 234, 81-113.

\bibitem[\protect\astroncite{Tao}{2006}]{tao2006nonlinear}
Tao, T. (2006).
\newblock  Nonlinear dispersive equations:  Local and global analysis.
\newblock  Number 106. {\em American Mathematical Soc.}

\bibitem[\protect\astroncite{Wang}{2012}]{wangyanjin2012}
Wang, Y. (2012).
\newblock Critical magnetic number in the magnetohydrodynamic Rayleigh-Taylor
  instability.
\newblock {\em Journal of Mathematical Physics}, 53(7).

\bibitem[\protect\astroncite{Wang and Xin}{2018}]{wangxinMHDST1}
Wang, Y. and Xin, Z. (2020).
\newblock Global Well-posedness of Free Interface Problems for the incompressible Inviscid Resistive MHD.
\newblock {\em arXiv 2009.11636}, preprint.

\bibitem[\protect\astroncite{Wu}{1997}]{wu1997LWPww}
Wu, S. (1997).
\newblock Well-posedness in Sobolev spaces of the full water wave problem in
  2-D.
\newblock {\em Inventiones mathematicae}, 130(1): 39--72.

\bibitem[\protect\astroncite{Wu}{1999}]{wu1999LWPww}
Wu, S. (1999).
\newblock Well-posedness in Sobolev spaces of the full water wave problem in
  3-D.
\newblock {\em Journal of the American Mathematical Society}, 12(2): 445--495.

\bibitem[\protect\astroncite{Zhang and Zhang}{2008}]{zhang2008free}
Zhang, P. and Zhang, Z. (2008).
\newblock On the free boundary problem of three-dimensional incompressible Euler equations.
\newblock {\em Communications on Pure and Applied Mathematics}, 61(7): 877--940.

\end{thebibliography}
\end{document}